\documentclass{amsart}
\usepackage[utf8]{inputenc}

\title[Multiplicativity and nonrealizable chain complexes]{Multiplicativity and nonrealizable equivariant chain complexes}
\author{Henrik Rüping and Marc Stephan}
\date{\today}

\usepackage{tikz} 
\usetikzlibrary{cd}
\usetikzlibrary{patterns,decorations}
\usepackage[utf8]{inputenc}
\usepackage{amsmath, amssymb}
\usepackage{amsthm, mathtools}
\usepackage[mathcal]{euscript}
\usepackage{tikz}
\usepackage{dsfont}
\usepackage{float}
\usepackage{enumitem}
\usetikzlibrary{matrix}
\usepackage[all]{xy}
\definecolor{darkblue}{rgb}{0,0,0.6}
\usepackage[ocgcolorlinks,colorlinks=true, citecolor=darkblue, filecolor=darkblue, linkcolor=darkblue, urlcolor=darkblue]{hyperref}
\usepackage[capitalize,noabbrev]{cleveref}

\newcounter{commentcounter}

\newcommand{\FF}{\mathbb{F}}
\newcommand{\ZZ}{\mathbb{Z}}
\newcommand{\QQ}{\mathbb{Q}}
\newcommand{\IN}{\mathbb{N}}
\newcommand{\IZ}{\mathbb{Z}}
\newcommand{\op}{{\mathrm{op}}}

\DeclareMathOperator{\id}{id}
\DeclareMathOperator{\gr}{gr}
\DeclareMathOperator{\Hom}{Hom}
\DeclareMathOperator{\Ann}{Ann}
\DeclareMathOperator{\Cone}{Cone}

\DeclareMathOperator{\Tor}{Tor}
\DeclareMathOperator{\aut}{aut}

\numberwithin{equation}{section}
\newtheorem{theorem}[equation]{Theorem}
\newtheorem{thmx}{Theorem}

\newtheorem{corollary}[equation]{Corollary}
\newtheorem{lemma}[equation]{Lemma}
\newtheorem{proposition}[equation]{Proposition}
\theoremstyle{definition}
\newtheorem{example}[equation]{Example}
\newtheorem{definition}[equation]{Definition}
\newtheorem{remark}[equation]{Remark}
\newtheorem{question}[equation]{Question}

\address{Continentale Krankenversicherung a.G., Ruhrallee 94, 44139 Dortmund, Germany}
\email{henrikrueping@googlemail.com}
\address{Faculty of Mathematics, Bielefeld University, PO Box 100 131, 33501 Bielefeld, Germany}
\email{marc.stephan@math.uni-bielefeld.de}

\thanks{The first author has been supported by ``KL2MG-interactions'' (no. 662400) awarded to Prof. Dr. Wolfgang Lück by the European Research Council. The second author is grateful to Max Planck Institute for Mathematics in Bonn for its hospitality and financial support.}
\begin{document}


\begin{abstract}
     Let $G$ be a finite $p$-group and $\FF$ a field of characteristic $p$. We filter the cochain complex of a free $G$-space with coefficients in $\FF$ by powers of the augmentation ideal of $\FF G$. We show that the cup product induces a multiplicative structure on the arising spectral sequence and compute the $E_1$-page as a bigraded algebra. As an application, we prove that recent counterexamples of Iyengar and Walker to an algebraic version of Carlsson's conjecture can not be realized topologically.
\end{abstract}
\maketitle

\section{Introduction}

Carlsson's conjecture \cite[Conjecture~I.3]{carlsson1986} in equivariant topology states that if $X\neq \emptyset$ is a finite, free $(\mathbb{Z}/p)^a$-CW complex, then $\sum_j \dim(H_j(X;\mathbb{F}_p))\ge 2^a$. An algebraic version \cite[Conjecture~II.2]{carlsson1986} (see also \cite{ademswan1995}) of this conjecture predicted the same bound for the total dimension of $H_*(C_*)$ for any finite, nonacyclic, free ${\mathbb{F}_p}(\mathbb{Z}/p)^a$-chain complex $C_*$. Recently, Iyengar and Walker constructed counterexamples to the algebraic version in \cite[Corollary~3.2]{iyengar2017examples} for odd $p$ and $a\ge 8$. This raises the question if they can be realized topologically by finite, free  $(\mathbb{Z}/p)^a$-CW complexes $X$ (see \cite[Remark~3.3]{iyengar2017examples}). These counterexamples are mapping cones built from a representative $w$ of a certain nonzero homology class in degree $2$ of the Koszul complex of $\FF_p (\mathbb{Z}/p)^a$  as below. We provide the following negative answer, even without the finiteness assumption on $X$ and for any $w$ of degree at least $2$. 

\begin{thmx}[Corollary~\ref{cor:mainthm}, Theorem~\ref{thm:realize}]\label{mthm} Let $G$ be an elementary abelian $p$-group of rank $a$ and $K_*$ be the Koszul complex of $\FF_p G$ with respect to a minimal set of generators of the augmentation ideal. Let $\Cone(w)_*$ denote the mapping cone of the map $\Sigma^r K_*\to K_*$ given by multiplication with a cycle $w$ of $K_*$ of degree $r$. Then the following are equivalent:
\begin{enumerate}
\item There is a topological space with a free $G$-action whose singular chain complex is quasi-isomorphic to $\Cone(w)_*$ as $\FF_p G$-chain complexes.
\item The cycle $w$ is a boundary or the degree of $w$ is at most $1$.
\end{enumerate}
\end{thmx}
We realize the mapping cone $\Cone(w)_*$ by the product of $S^{r+1}$ and the $a$-torus $T^a$ if $w$ is a boundary, by $S^3\times T^{a-1}$ if $w$ represents a nonzero homology class of degree $1$, and by the empty space if $w$ represents a nonzero homology class of degree $0$. The $G$-actions are described in Section~\ref{sec:rmappingcones} and for $S^3\times T^{a-1}$ established by proving that restriction of scalars induces an $\aut(\FF_p G)$-action on the Koszul homology $H_*(K_*)$ in Proposition~\ref{prop:actiononkoszulhomology}. In particular, the total rank of $H_*(\Cone(w)_*)$ is at least $2^a$, whenever the mapping cone is topologically realizable by a nonempty space. Thus there are no counterexamples to Carlsson's conjecture among the chain complexes of the form $\Cone(w)_*$.

Questions concerning topological realizations of equivariant chain complexes have a long history with both positive and negative results. Steenrod asked in \cite[Problem~51]{lashof1965Steenrodproblem} which finitely generated $\ZZ\pi$-modules for a finite group $\pi$ can be realized as the reduced integral homology of a finite $\pi$-CW complex concentrated in a given degree. Variants of this existence problem for equivariant Moore spaces drop the finiteness assumption, ask for free actions, or change the ring of coefficients.

In \cite{swan1969}, Swan produced the first module which can not be realized by a finite, equivariant Moore space, though it can be realized by an infinite one. 

Arnold established the existence of equivariant Moore spaces for any finitely generated $\ZZ\pi$-module when $\pi$ is a cyclic group of prime power order in \cite{arnold1977}. 

In \cite{carlsson1981counterexample}, Carlsson proved that for any abelian, noncyclic group there exist modules with no equivariant Moore space -- even no infinite one -- using group cohomology and the action of the Steenrod algebra.

Benson and Habegger provided a more conceptual interpretation of Carlsson's result in \cite{bensonhabegger1987steenrod} showing that realizability yields restrictions for the structure of associated support varieties, and produced more examples of nonrealizable modules. 

In \cite{smith1986realization}, Smith constructed an obstruction theory to the question whether a given $\ZZ\pi$-chain complex is homotopy equivalent to chains on a simply-connected, free $\pi$-space, and proved that rationally any finite dimensional, free $\QQ\pi$-chain complex is stably realizable. 

In \cite{hanke2004realization}, Hanke constructed a Poincar\'e duality space equipped with a $\ZZ/p$-action based on realizing an equivariant chain complex to answer a question of Sikora about the vanishing of a differential in the Leray-Serre spectral sequence.

To show that a chain complex can not be realized as a chain complex $C_*(X;\FF_p)$ for an equivariant space $X$, one may consider the dual cochain complex $C^*(X;\FF_p)$ and try to use the additional multiplicative structure. Nonequivariantly, this does not provide any obstructions. Indeed, any nonnegatively graded $\FF_p$-chain complex (with nontrivial $H_0$) can be realized by a wedge of spheres. Instead, the interplay between the group action and additional structure has to be taken into account. Carlsson's approach to nonexistence of certain equivariant Moore spaces can be adapted to realizability questions of equivariant chain complexes, but does not provide obstructions for realizing $\Cone(w)_*$ (see Remark~\ref{rem:steenrodapproach}). To establish nonrealizability of $\Cone(w)_*$, we do not consider the interaction between the group cohomology and the Steenrod algebra as in \cite{carlsson1981counterexample}, but the interaction of group cohomology and the multiplicative structure in the following spectral sequence. Its direct computation from the definition requires only algebraic data and is easy. The full identification of its $E_1$-page involves Jennings bases for finite $p$-groups $G$ which we will recall in Section~\ref{sec:jennings}. Here we just note that in the elementary abelian case, all elements of a Jennings basis have degree one.
\begin{thmx}[{Theorem~\ref{thm:main1}}]\label{mthm:1} Let $I\subset {\FF}G$ be the augmentation ideal for a finite $p$-group $G$ and field $\FF$ of characteristic $p$. Let $L$ be the nilpotency index of $I$ minus one.
\begin{enumerate}[leftmargin=*]
\item \label{mthm:1i}Filtering the ${\FF}$-cochains of a free $G$-space $X$ by multiplication with the power $I^{L-k}$ of the augmentation ideal $I$ in filtration degree $k$
gives rise to a natural, multiplicative spectral sequence $\{d_r\colon E_r^{k,q}\rightarrow E_r^{k-r,q+1}\}$ converging to $H^*(X;\FF)$ whose $E_1$-term is as a bigraded algebra naturally isomorphic to 
\[E_1^{k,q}\cong H^q(X/G;\FF)\otimes \gr_\cup^{k}(C^0(G)),\]
where $\gr_\cup(C^0(G))$ denotes the graded ring associated to filtering the $\FF$-valued functions ${\FF}^G=C^0(G)$ on $G$ equipped with the pointwise multiplication by powers of the augmentation ideal $I^{L-k}$.
\item \label{mthm:1ii}This graded ring is isomorphic to ${\FF}[y_1,\ldots,y_a]/(y_1^p,\ldots,y_a^p)$, where $p^a$ is the order of $G$, and the degrees $\alpha(j)$ of the classes $y_j$ correspond to the degrees of the elements of a Jennings basis.
\item \label{mthm:1iii} For classes $y_j$ with $\alpha(j)=1$, the differential $d_1$ is given by 
\[d_1(x\otimes y_j) = (-1)^{|x|} x\cup a_j,\]
where the $a_j$'s form an explicit basis of $H^1(BG;{\FF})$ and $H^*(BG;\FF)$ acts on $H^*(X/G;\FF)$ through a choice of a classifying map $X/G\to BG$.
\item \label{mthm:1iv} Furthermore, multiplication with a group ring element of the form $g-1\in I^k$ induces a derivation $E^{*,*}_r\rightarrow E^{*-k,*}_r$ on each page which only depends on $[g-1]\in \frac{I^k}{I^{k+1}}$.
\item \label{mthm:1v}These derivations for $[g-1]\in I/ I^2$ and $d_1(y_j)$ when $\alpha(j)=1$ already determine the entire differential on the $E_1$-page.
\end{enumerate}
\end{thmx}

We only need the spectral sequence for elementary abelian $p$-groups to provide obstructions to realizability, but we believe that the spectral sequence and the calculation of the associated graded $\gr_\cup(C^0(G))$ are interesting for finite $p$-groups $G$ in general.

For $p=2$ and $G$ an elementary abelian $2$-group, this spectral sequence agrees with the Eilenberg-Moore spectral sequence as we explain in Remark~\ref{rem:eilenbergmoore}. While we do not establish interesting special cases of Carlsson's conjecture through the spectral sequence, we provide a further link to a weak form of the Buchsbaum-Eisenbud-Horrocks conjecture shedding light on the difficulty of Carlsson's conjecture in Remark~\ref{rem:buchsbaumeisenbudhorrocks}.

A homological version of the spectral sequence from \cref{mthm:1} for arbitrary groups is studied in \cite{papadimasuciu2010spectralsequence} with an analogous, additive identification of the $E_1$-page, but in there comultiplicativity of the spectral sequence is not considered.

In \cite{puppe2009}, Puppe considered different multiplicative aspects related to Carlsson's conjecture.

The paper is organized as follows. In Section~\ref{sec:conv} we introduce some notation and sign conventions regarding cochain complexes.
In Section~\ref{sec:filtration} we define the filtration on a free ${\FF}G$-cochain complex $C^*$, and express the pages $E_0$ and $E_1$ of the induced spectral sequence in terms of the ${\FF}G$-module structure on the cochain complex.
In Section~\ref{sec:mult} we prove that in the case where $C^*$ arises as the cochains of a free $G$-space, the $\cup$-product induces a multiplicative structure on the associated spectral sequence. Moreover, we compute $E_0$ and $E_1$ as algebras. To complete the computation, we determine the ring $C^0(G)={\FF}^G$ of $\FF$-valued functions on $G$ under pointwise multiplication and its associated graded ring $\gr_\cup(C^0(G))$ in Section~\ref{sec:jennings} with the help of Jennings bases. 
In Section~\ref{sec:d1} we establish a general strategy to compute $d_1$ and prove Theorem~\ref{mthm:1}.

In Section~\ref{sec:nrmappingcones} we use all the structure of the spectral sequence established in Theorem~\ref{mthm:1} to show that the counterexamples of Iyengar and Walker from \cite[Theorem~3.1, Corollary~3.2]{iyengar2017examples} and more generally $\Cone(w)_*$ for a representative $w$ of a nonzero homology class of degree at least $2$ can not be realized as the cochains on a free $(\ZZ/p)^a$-space. In Section~\ref{sec:rmappingcones} we prove that for all other cycles $w$, the mapping cone $\Cone(w)_*$ can be realized topologically.
This completes the proof of Theorem~\ref{mthm}.

\subsection*{Acknowledgements}
It is our pleasure to thank Jesper Grodal, Bernhard Hanke, Markus Hausmann, Srikanth Iyengar, Achim Krause, Sune Precht Reeh, and Erg{\" u}n Yal\c c\i n  for helpful discussions. We thank a referee for detailed comments.

\section{Conventions and formulas for cochain complexes}\label{sec:conv}
For the whole paper, we fix a prime $p>0$, a finite $p$-group $G$ with $p^a$ elements and a field $\FF$ of characteristic $p$. 
The augmentation ideal of a finite group algebra over a field is nilpotent if and only if the field is of prime characteristic $p>0$ and the group is a $p$-group; see \cite[Theorem~9]{connell1963} for a more general result.

Let $L$ be the largest integer such that $I^L\neq 0$, where $I\subset {\FF}G$ is the augmentation ideal. Thus $L+1$ is the nilpotency index of $I$. It is easy to see that $L$ is independent of the chosen coefficient field $\FF$ since a power $I^k$ is nonzero if and only if there exist group elements $g_1,\ldots,g_k$ such that $(g_1-1)\cdot\ldots\cdot(g_k-1)\neq 0$.

If unspecified, homology and cohomology are computed with $\FF$-coefficients.

We will entirely work with simplicial sets, in particular, $G$-space means simplicial set with a simplicial $G$-action. The normalized simplicial cochain complex of a simplicial set $X$ with coefficients in ${\FF}$ will be denoted by $C^*(X)$. More precisely, let $C_*(X)$ denote the normalized chain complex of $X$ with coefficients in ${\FF}$ and ${\FF}[0]$ the chain complex with ${\FF}$ concentrated in degree $0$. Then $C^*(X)$ is the hom-complex $\Hom(C_*(X),{\FF}[0])$ with the usual regrading. In particular, the differential is given by $d(f)=-(-1)^{|f|} f\circ d$. In general, we follow the sign conventions of \cite{mcclure2003multivariable}. For instance, the cup product 
\[\cup\colon C^*(X)\otimes C^*(X)\to C^*(X)
\]
is defined on cochains $\varphi\in C^k(X), \psi\in C^l(X)$ as
\[(\varphi\cup \psi)(\sigma) =(-1)^{kl} \varphi(\sigma|_{[0,\ldots,k]})\psi(\sigma|_{[k,\ldots,k+l]})
\]
on simplices $\sigma$ of dimension $k+l$. If $X$ has a simplicial, left $G$-action, then $C^*(X)$ inherits a right $G$-action via
\[(\psi.g)(\sigma)\coloneqq \psi(g\sigma). \]

The cochain formula above for the cup product gives a way to compute the multiplication on the cohomology of a simplicial set; for two given cohomology classes pick representatives, compute their cup product as above and take the cohomology class given by their cochain level product.

The surjection operad \cite{mcclure2003multivariable} provides similar cochain level formulas for additional operations and in particular for $\cup_i$-products. So one could for example use it to compute Steenrod squares.  While in this article, we only need formulas for $\cup_1$-products, it may be interesting to consider the entire action of the surjection operad on $C^*(X)$. Therefore, we use their formulas already in here.

We need the $\cup_1$-product in Lemma~\ref{lem:cupgradcomm} to keep track of the difference between $\varphi \cup \psi$ and $\psi \cup \varphi$ on the cochain level. For $\varphi\in C^k(X), \psi\in C^l(X)$, it is defined by the formula 
\[(\varphi\cup_1 \psi)(\sigma) =\sum_{i=0}^{k-1} (-1)^{il+k+i} \varphi(\sigma|_{[0,\ldots, i,i+l,\ldots,k+l-1]})\cdot \psi(\sigma|_{[i,\ldots,i+l]})\]
from \cite[Definition~2.2]{mcclure2003multivariable}. The $\cup_1$-product is an element in the hom-complex $\Hom(C^*(X)\otimes C^*(X),C^*(X))$ -- natural in the simplicial set $X$ -- such that
\[d_{\Hom}(\cup_1)(\varphi\otimes \psi)=(-1)^{kl}\psi\cup \varphi-\varphi\cup \psi.\] By the definition of the boundary in the hom-complex, the left-hand side equals $d(\varphi\cup_1 \psi)+ (d\varphi)\cup_1\psi+(-1)^k\varphi\cup_1(d\psi)$. 

\begin{remark}\label{rem:propofcup1}
We will use that $\varphi \cup_1 \psi=0$ whenever the degree $k$ of $\varphi$ is zero by definition of $ \cup_1$ as the sum is empty. It also vanishes if the degree of $\psi$ is zero since $\varphi$ is a normalized cochain that gets evaluated in the degenerate simplices $\sigma|_{[0,\ldots, i,i,\ldots,k-1]}$.
\end{remark}

\section{The filtration and its associated spectral sequence}
\label{sec:filtration}

Since $I^k\cdot I^{L-k+1} = I^{L+1}=0$, we see that $I^{L-k+1}$ lies in the (right) annihilator $\Ann(I^k)$ of $I^k$. The converse inclusion is much harder to establish, but holds as well.
\begin{theorem}[{\cite[Theorem~11]{hill1970}}]
\label{thm:hillanni}
We have in ${\FF}G$ that $\Ann(I^k)=I^{L-k+1}$.
\end{theorem}
\begin{remark}\label{rem:hilljennigsarbf}
Hill's result above is based on the work \cite{jennings1941structure} of Jennings which we will use in Section~\ref{sec:jennings}. Strictly speaking, both papers are written working over the field $\FF_p$, but the results extend to arbitrary fields $\FF$ of characteristic $p$ by tensoring with $-\otimes_{\FF_p} \FF$.
\end{remark}

By Theorem~\ref{thm:hillanni}, the two filtrations on a free, right ${\FF}G$-cochain complex $C^*$ in the following definition agree.

\begin{definition}\label{def:filtr}
We filter a free ${\FF}G$-cochain complex $C^*$ via 
\[F^kC^* = \{x \in C^* \mid  x.\lambda = 0\; \forall \lambda\in I^{k+1}\} = C^*.I^{L-k}\]
to obtain an increasing filtration
\[0=F^{-1}C^*\subset F^0C^*\subset F^1C^* \subset \ldots \subset F^LC^*=C^*.
\]
\end{definition}

The filtration is chosen so that the unit for the arising multiplicative spectral sequence in Section~\ref{sec:mult} lies in bidegree $(0,0)$. This differs from the indexing conventions in the analogous homological spectral sequence from \cite{papadimasuciu2010spectralsequence}. Moreover, while it often happens that the cohomology of the  subquotients $F^{k+1}C^*/F^kC^*$ of a filtered cochain complex $\{F^kC^*\}$ can be identified with some cohomology shifted by $k$, this is not the case in our situation. Therefore, we change the indexing convention by refraining from replacing $C^q$ by $C^{k+q}$. In particular, the subquotients of $H^q(C^*)$  on the $E_\infty$-page lie in the $q$-th row instead of the $q$-th diagonal.

\begin{definition}\label{def:erkq} The pages of the spectral sequence arising from the filtration $F^*C^*$ are defined as 
\[E^{k,q}_r \coloneqq \frac{\{c\in F^kC^q\mid dc\in F^{k-r}C^{q+1}\}}{\{c\in F^{k-1}C^q \mid dc\in F^{k-r}C^{q+1}\} +d(F^{k+r-1}C^{q-1})\cap F^kC^q},\]
and the boundary map induces a differential 
\[d_r\colon E_r^{k,q}\rightarrow E_r^{k-r,q+1}.\]
Furthermore, the subquotients of the induced filtration on $H^*(C^*)$ are 
\[0\rightarrow F^{k-1}H^q(C^*)\rightarrow F^{k}H^q(C^*)\rightarrow E_\infty^{k,q}\rightarrow 0.\]
\end{definition}

Since in our situation the filtration has finite length $L$, the differential 
$d_r$ vanishes for $r\ge L+1$. Thus the spectral sequence converges to $H^*(C^*)$.

\begin{remark}\label{rem:multmaps} The maps 
\[F^lC^* \otimes_{{\FF}G} I^k \rightarrow F^{l-k}C^*, \quad c\otimes \lambda \mapsto c.\lambda\]
induce maps
\[ \frac{F^{l+1}C^*}{F^{l}C^*}\otimes_{{\FF}G} \frac{I^k}{I^{k+1}}
\rightarrow \frac{F^{l+1-k}C^*}{F^{l-k}C^*}.\] 
Thus each $\lambda \in I^k$ induces a map of filtered cochain complexes $F^lC^*\to F^{l-k}C^*$ such that the induced map of spectral sequences
\[E_*^{*,*}\rightarrow E_*^{*-k,*}\]
depends only on the class $[\lambda]\in \frac{I^k}{I^{k+1}}$. This equips the spectral sequence with a $\gr({\FF}G,\cdot)$-module structure, where $\gr({\FF}G,\cdot)$ denotes the associated graded of the group ring ${\FF}G$ with its multiplication $\cdot$ filtered by powers of $I$.
\end{remark}

We compute $E_0$ and $E_1$ in terms of the maps from Remark~\ref{rem:multmaps} for $l+1=L$. Note that if we tensor two ${\FF}G$-modules with trivial $G$-action over ${\FF}G$, we can just tensor them over ${\FF}$ instead. 
\begin{proposition}\label{prop:rightmulone1} Let $C^*$ be a free ${\FF}G$-cochain complex. The natural map 
\[\frac{C^q}{C^q.I}\otimes_{{\FF}} \frac{I^{k}}{I^{k+1}}\rightarrow \frac{F^{L-k}C^q}{F^{L-k-1}C^q},\quad [c]\otimes [\lambda]\mapsto [c.\lambda], \]
is an isomorphism. Thus the $E_0$-page is given by
\[E_0^{L-k,q}\stackrel{\cong}{\leftarrow} E_0^{L,q}\otimes_{{\FF}} \frac{I^k}{I^{k+1}}.\]
The column $E_0^{L,*}$ is the quotient $C^*\otimes_{{\FF}G}{\FF}$ and determines the differential $d_0: E_0^{L-k,q}\rightarrow E_0^{L-k,q+1}$ via $d_0([c.\lambda])=d_0([c]).\lambda$. 
\end{proposition}
\begin{proof}
Using that $C^q$ is a free ${\FF}G$-module, we can reduce it to the case where $C^q={\FF}G$. In this case the map is the composite
\[\frac{{\FF}G}{{\FF}G.I}\otimes_{\FF}\frac{I^k}{I^{k+1}}\cong{\FF}\otimes_{\FF}\frac{I^k}{I^{k+1}}\cong \frac{I^k}{I^{k+1}}. \]

Inserting the definition $E_0^{L-k,*}=\frac{C^*.I^{k}}{C^*.I^{k+1}}$ yields the second statement. By definition, $E_0^{L,*}$ is the quotient $C^*\otimes_{{\FF}G}{\FF}$. Since the differential on $C^*$ is a map of ${\FF}G$-modules and thus commutes with multiplication with group ring elements, we deduce the third statement about the differential. 
\end{proof}

Thus every column on the $E_0$-page is a direct sum of copies of the last column. That last column is the cochain complex $C^*\otimes_{{\FF}G}{\FF}$. Hence we obtain for the next page:

\begin{corollary}\label{cor:e1byrightaction} The $E_1$-page is naturally isomorphic to
\[E_1^{L-k,q} \cong H^q\left(\frac{C^*}{C^*.I}\right)\otimes \frac{I^{k}}{I^{k+1}},\]
and the isomorphism $E_0^{L,*}\otimes \frac{I^k}{I^{k+1}} \rightarrow E_0^{L-k,*}$ induces an isomorphism
\[E_1^{L,*}\otimes \frac{I^k}{I^{k+1}} \rightarrow E_1^{L-k,*},\quad [c]\otimes [\lambda]\mapsto [c.\lambda].\]
In particular, any $\FF G$-homotopy equivalence between free $\FF G$-cochain complexes induces an isomorphism of spectral sequences on pages $E_{\geq 1}^{*,*}$.
\end{corollary}

The following cochain complex is the smallest example of the nonrealizable cochain complexes that we will consider in Section~\ref{sec:nrmappingcones}. It also shows that while $E_r^{L,*}\otimes \frac{I^k}{I^{k+1}} \cong E_r^{L-k,*}$ for $r\leq 1$, there is no such isomorphism for $r\geq 2$.

\begin{example}\label{ex:smallnonrealizable} Let $p=2$ and $G=(\ZZ/2)^2$ an elementary abelian $p$-group of rank $2$ with generators $f_1$, $f_2$. Denoting $\lambda_i\coloneqq f_i-1$, the powers of the augmentation ideal are $I=(\lambda_1,\lambda_2)$, $I^2=(\lambda_1\lambda_2)$, $I^3=0$, and thus $L=2$. Consider the free $\FF_2(\ZZ/2)^2$-cochain complex $C^*$ given by
\[\FF_2G \stackrel{\begin{pmatrix}  \lambda_1 \\  \lambda_2\end{pmatrix}}{\longrightarrow} (\FF_2 G)^2 \stackrel{\begin{pmatrix}\lambda_2 & \lambda_1\end{pmatrix}}{\longrightarrow} \FF_2G \stackrel{\begin{pmatrix}\lambda_1\lambda_2\end{pmatrix}}{\longrightarrow}
\FF_2G \stackrel{\begin{pmatrix}\lambda_1 \\ \lambda_2\end{pmatrix}}{\longrightarrow} (\FF_2 G)^2 \stackrel{\begin{pmatrix}\lambda_2 & \lambda_1\end{pmatrix}}{\longrightarrow} \FF_2G.\]

The differential of the quotient $C^*/C^*.I$ is zero. The $E_1$-page consists of the three columns
\begin{align*}
     E^{0,*}_1&\cong C^*/C^*.I\otimes \FF_2\lambda_1\lambda_2,\\
     E^{1,*}_1&\cong C^*/C^*.I\otimes (\FF_2\lambda_1\oplus \FF_2\lambda_2),\\
     E^{2,*}_1&\cong C^*/C^*.I\otimes \FF_2.
\end{align*}
Since the differentials of the spectral sequence are induced by the differential of $C^*$, it is straightforward to calculate the following $E_1$- and $E_2$-page. The labels of the arrows denote the rank of the corresponding differential.
\begin{figure}[H]
\centering
\begin{tikzpicture}
  \matrix (m) [matrix of math nodes,
    nodes in empty cells,nodes={minimum width=5ex,
    minimum height=5ex,outer sep=-5pt},
    column sep=1ex,row sep=1ex]{
        \phantom{3}& & & & \\
          6     &   0      &   0   &  0    \\
          5     &   \FF_2  & \FF_2^2   &  \FF_2    \\
          4     &   \FF_2^2& \FF_2^4   &  \FF_2^2    \\
          3     &   \FF_2  & \FF_2^2   &  \FF_2    \\
          2     &  \FF_2   & \FF_2^2 &  \FF_2    \\
          1     &  \FF^2_2 & \FF_2^4  &\FF_2^2   \\
          0     &  \FF_2   & \FF_2^2 & \FF_2    \\
    \quad\strut &   0  &  1  &  2  & \strut\\ };
\draw[thick,->] (m-9-1.east)node [left,yshift=-0.2cm,align=left]{{$E_1^{k,q}$}}--(m-1-1.east) node [left]{{$q$}};
\draw[thick,->] (m-9-1.north)--(m-9-5.north) node [right]{{$k$}};
\draw[-stealth] (m-8-3.north west) -- (m-7-2.south east)node [right,align=center,midway]{{${\scriptstyle 2}$}};
\draw[-stealth] (m-8-4.north west) -- (m-7-3.south east)node [right,align=center,midway]{{${\scriptstyle 1}$}};
\draw[-stealth] (m-7-3.north west) -- (m-6-2.south east)node [right,align=center,midway]{{${\scriptstyle 1}$}};
\draw[-stealth] (m-7-4.north west) -- (m-6-3.south east)node [right,align=center,midway]{{${\scriptstyle 2}$}};
\draw[-stealth] (m-5-3.north west) -- (m-4-2.south east)node [right,align=center,midway]{{${\scriptstyle 2}$}};
\draw[-stealth] (m-5-4.north west) -- (m-4-3.south east)node [right,align=center,midway]{{${\scriptstyle 1}$}};
\draw[-stealth] (m-4-3.north west) -- (m-3-2.south east)node [right,align=center,midway]{{${\scriptstyle 1}$}};
\draw[-stealth] (m-4-4.north west) -- (m-3-3.south east)node [right,align=center,midway]{{${\scriptstyle 2}$}};
\end{tikzpicture}
\begin{tikzpicture}
  \matrix (m) [matrix of math nodes,
    nodes in empty cells,nodes={minimum width=5ex,
    minimum height=5ex,outer sep=-5pt},
    column sep=1ex,row sep=1ex]{
        \phantom{3}& & & & \\
          6     &   0      &   0   &  0    \\
          5     &   0      & 0         &  \FF_2    \\
          4     &   0      & \FF_2^2   &  0          \\
          3     &   \FF_2  & 0         &  0        \\
          2     &  0       & 0        &  \FF_2    \\
          1     &  0       & \FF_2^2  &0         \\
          0     &  \FF_2   & 0        & 0        \\
    \quad\strut &   0  &  1  &  2  & \strut \\};
\draw[thick,->] (m-9-1.east)node [left,yshift=-0.2cm,align=left]{{$E_2^{k,q}$}}--(m-1-1.east) node [left]{{$q$}};
\draw[thick,->] (m-9-1.north)--(m-9-5.north) node [right]{{$k$}};
\draw[-stealth] (m-6-4.north west) -- (m-5-2.south east)node [above,align=center,near start]{{${\scriptstyle 1}$}};
\end{tikzpicture}
\caption{The $E_1$- and $E_2$-page arising from the cochain complex given above.}
\label{figexample}
\end{figure}
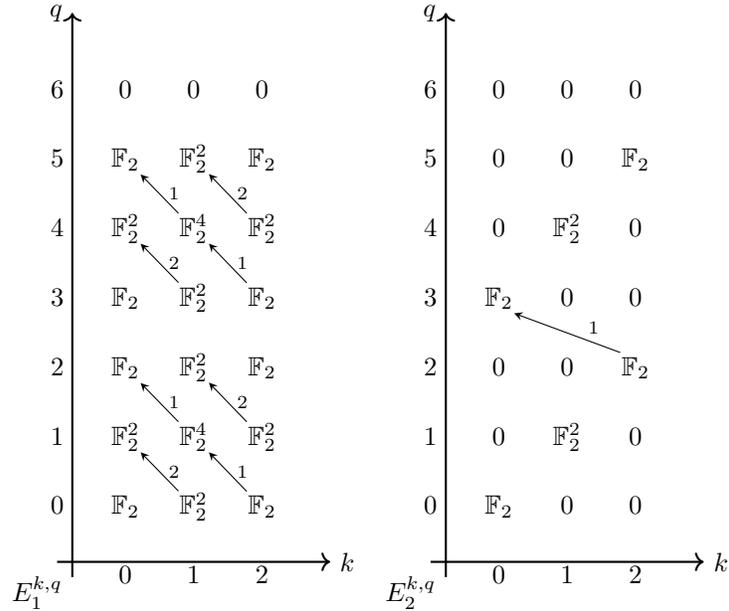
\end{example}

\section{Multiplicative structure}\label{sec:mult}
For the cochain complex $C^*=C^*(X)$ of a free $G$-space $X$, we will equip the spectral sequence of Section~\ref{sec:filtration} with a multiplicative structure, and calculate the $E_1$-page as a tensor product of $H^*(X/G)$ with an associated graded ring of $C^0(G)$, the $0$-cochains on $G$ with the cup product.

\begin{lemma} \label{lem:multsandfiltr} Equip the tensor product (over ${\FF}$) of two right ${\FF}G$-modules with the diagonal $G$-action $(x\otimes y).g = x.g\otimes y.g$. Let $\lambda_g\coloneqq g-1$. Then we have:
\[(x\otimes y).\lambda_g = x.\lambda_g\otimes y.\lambda_g+x.\lambda_g\otimes y+x\otimes y.\lambda_g .\]
\end{lemma}
\begin{proof} The equation follows immediately from expanding the terms:
\begin{align*}
    (x\otimes y).\lambda_g&=x.g\otimes y.g - x\otimes y,\\
    x.\lambda_g\otimes y.\lambda_g&=(x.g)\otimes (y.g) - (x.g)\otimes y-x\otimes (y.g)+x\otimes y,\\
    x.\lambda_g\otimes y &= (x.g) \otimes y-x\otimes y,\\
    x\otimes y.\lambda_g &= x \otimes (y.g)-x\otimes y.\qedhere 
    \end{align*}
\end{proof}

Just as $\cup_1$, higher $\cup_i$-products can be defined with the surjection operad; see \cite[Remark~2.11]{mcclure2003multivariable}.
\begin{lemma}\label{lem:filtrismult} The operators $\lambda_g :C^*(X)\rightarrow C^*(X)$ satisfy the following twisted Leibniz rule:
\[(\varphi\cup \psi).\lambda_g = \varphi.\lambda_g \cup \psi.\lambda_g +\varphi.\lambda_g \cup \psi + \varphi\cup \psi.\lambda_g.\]
The analogous twisted Leibniz rule holds for higher $\cup_i$-products as well.
\end{lemma}
\begin{proof}
The twisted Leibniz rule holds for any $G$-equivariant map $C^*\otimes C^*\rightarrow C^*$ by Lemma~\ref{lem:multsandfiltr}. The cup product and higher cup products are such $G$-equivariant maps, since they are natural. 
\end{proof}

We will establish multiplicativity based on \cref{lem:filtrismult} and the following remark.
\begin{remark}\label{rem:filtrationg-1} Recall that Definition~\ref{def:filtr} provides two coinciding definitions of the filtration. For any cochain $x$ of a free $\FF G$-cochain complex $C^*$, we have by the definition of the filtration via annihilators that:
\begin{align*}x\in F^{k+1}C^*&\Leftrightarrow x.\lambda =0 \mbox{ for all }\lambda \in I^{k+2}\\
&\Leftrightarrow x.\lambda\lambda' =0 \mbox{ for all }\lambda\in I,\lambda' \in I^{k+1}\\
&\Leftrightarrow x.\lambda \in F^{k}C^* \mbox { for all }\lambda\in I\\
&\Leftrightarrow x.(g-1) \in F^{k}C^* \mbox { for all }g\in G.
\end{align*}
\end{remark}

\begin{proposition}\label{prop:filtismult} For the cochain complex $C^*=C^*(X)$ of a free $G$-space $X$, the filtration $\{F^kC^*\}$ is  multiplicative, i.e. it restricts to 
\[\cup: F^kC^* \otimes F^l C^* \rightarrow F^{k+l}C^*.\]
Thus the spectral sequence of the filtered cochain complex is multiplicative. In particular, the $\cup$-product induces multiplications
\[E_r^{k,q}\otimes E_r^{k',q'}\rightarrow E_r^{k+k',q+q'}\]
such that each differential is a derivation.
\end{proposition}
\begin{proof}
Let $\varphi\in F^kC^*, \psi \in F^lC^*$ be two arbitrary elements. We prove that the filtration is multiplicative by induction on $k+l$. By Remark~\ref{rem:filtrationg-1}, it suffices to show that $(\varphi\cup \psi).\lambda_g\in F^{k+l-1}C^*$ for all $g\in G$. This follows, since the three summands in the twisted Leibniz rule from \cref{lem:filtrismult} are contained in $F^{k+l-1}C^*$ by induction hypothesis.

A multiplicative filtration of a differential graded algebra induces a multiplicative spectral sequence by \cite[Theorem~2.14]{mccleary2001user}.
\end{proof}

Using the multiplication, we find a second way of describing the pages $E_0$ and $E_1$ out of $C^*(G)$ and $C^*(X/G)$. Since $G$ acts freely on $X$,  we have a principal $G$-bundle $X\to X/G$. For a total space $X$ such that the base $X/G$ is reduced, Brown's Theorem about twisted cartesian products (see e.g. \cite[\S31]{maysimplicialobjects}) provides a chain homotopy equivalence from $C_*(X)$ to the tensor product $C_*(X/G)\otimes_t C_*(G)$ with a twisted differential. As in our situation $G$ is a finite group, we will observe that the cup product induces an isomorphism
\[C^*(X/G)\otimes C^*(G)\cong C^*(X)
\]
of graded $C^*(X/G)$-$\FF G$-bimodules so that we will just transfer the differential from $C^*(X)$ without having to calculate a twisting cochain explicitly.

We recall standard combinatorial models for the universal principal $G$-bundle $EG\to BG$. Our definitions agree with the ones from \cite[Chapter ~IV\S19]{maysimplicialobjects} for the opposite group, that is, $EG=WG^\op$ and $BG=\overline{W}G^{\op}$.

\begin{definition} Let $BG$ denote the simplicial set whose set of $m$-simplices is the set of all $m$-tuples $(g_1,\ldots,g_m)\in G^m$. The degeneracy maps are given by inserting the neutral element, i.e.
\[s_i((g_1\ldots,g_m))\coloneqq(g_1,\ldots,g_{i},e,g_{i+1},\ldots,g_m)\]
for $i=0,\ldots,m$. The face maps $d_0$ and $d_m$ are given by leaving out the first, resp. last entry in $(g_1,\ldots,g_m)$. The other face maps are given by
\[d_i((g_1\ldots,g_m))\coloneqq(g_1,\ldots,g_{i}g_{i+1},\ldots,g_m).\]
\end{definition}

So we see that the nondegenerate simplices are exactly those tuples which do not contain the neutral element.

We calculate the multiplicative structure of $C^*(BG)$.
\begin{lemma}\label{lem:cbgfreeal} The underlying algebra of $C^*(BG)$ is a free noncommutative algebra (i.e. tensor algebra) on $C^1(BG)$.
\end{lemma}
\begin{proof}
The elements $g\in G\setminus \{e\}$ form a basis of $C_1(BG)$. Let $g^*\colon C_1(BG)\to {\FF}$ denote the corresponding dual basis elements of $C^1(BG)$. It suffices to show that the products $g_1^*\cup \ldots \cup g_m^*$  for $(g_1,\ldots,g_m)\in (G\setminus\{e\})^m$ form a basis of $C^m(BG)$. We compute using the definition of the cup product:
\[g_1^*\cup g_2^*\cup \ldots \cup g_m^*(h_1,\ldots, h_m) =\begin{cases} (-1)^{m(m-1)/2} \quad &\text{ if } g_i=h_i \text{ for all } i\\
0 \quad & \text{ otherwise.}
\end{cases}
\]
It follows immediately that the products $g_1^*\cup \ldots \cup g_m^*$ are linearly independent. They form a basis since $C^m(BG)$, i.e. the vector space of functions $ (G\setminus\{e\})^m\to{\FF}$, is of dimension $(|G|-1)^m$.
\end{proof}

The space $BG$ is the base of a principal $G$-bundle with total space defined as follows.
\begin{definition} Let $EG$ denote the simplicial set whose $m$-simplices are given by all $(m+1)$-tuples $(g_0,g_1,\ldots,g_m)$. The degeneracies are given by
\[s_i(g_0,\ldots,g_m) =(g_0,\ldots, g_{i},e,g_{i+1},\ldots,g_m)\]
for $0\le i \le m$. The face maps are given by
\begin{align*}
d_m(g_0,\ldots,g_m) &=(g_0,\ldots,g_{m-1}),\\
d_i(g_0,\ldots,g_m) &=(g_0,\ldots,g_ig_{i+1},\ldots g_m) \end{align*}
for $i<m$.
\end{definition}

\begin{remark} $EG$ carries a free, simplicial, left $G$-action via \[g\cdot (g_0,\ldots,g_m) = (gg_0,g_1,\ldots,g_m).\]
The quotient is exactly $BG$. The fact that $G$ acts on the first coordinate makes this entry special. Indeed, for given $(g_1,\ldots,g_m)$ the set $\{(g_0,\ldots,g_m)\mid g_0\in G\}$ consists of all lifts of that simplex.  The first coordinate just specifies one lift.

On cochains the map $C^*(BG)\rightarrow C^*(EG)$ sends an $m$-cochain $\varphi$ to the $m$-cochain 
\[(g_0,\ldots,g_m)\mapsto \varphi(g_1,\ldots,g_m).\]
So its image consists of all cochains that are independent of the first coordinate. 
\end{remark}
Recall that the left $G$-action on $EG$ induces a right action on the cochain complex $C^*(EG)$. 
\begin{lemma}\label{lem:CSTAREGISCSTARBGTENSORC0EG} The cup product induces an isomorphism of graded $C^*(BG)$-${\FF}G$-bimodules:
\[\cup\colon C^*(BG)\otimes_{{\FF}}C^0(EG)\cong C^*(EG).\]
\end{lemma}\begin{proof}
The nondegenerate $m$-simplices of $EG$ are the tuples $(h_0,\ldots,h_m)$ such that $(h_1,\ldots,h_m)$ is nondegenerate in $BG$, i.e. $h_i\neq e$ for all $i>0$. Consider a cochain $\varphi=g_1^*\cup \ldots \cup g_m^*$ in $C^m(BG)\subset C^m(EG)$ as in the basis in the proof of Lemma~\ref{lem:cbgfreeal} and $g_0^*\in C^0(EG)$ for some element $g_0\in G$. By definition of the cup product, we have:
\[(\varphi\cup g_0^*)(h_0,\ldots,h_m) = \pm \varphi(h_1,\ldots,h_m) g_0^*(h_0\ldots h_m).
\]
It follows that $g_1^*\cup \ldots \cup g_m^*\cup g_0^*$ with $g_i\in G\setminus\{e\}$ for $i>0$ and $g_0\in G$ is a basis for $C^m(EG)$; this map sends only the tuple $(g_0g_m^{-1}\ldots g_1^{-1},g_1,\ldots,g_m)$ to $\pm 1$ and all other tuples to zero. Thus the cup product induces an isomorphism $ C^m(BG)\otimes_{{\FF}}C^0(EG)\to C^m(EG)$. It is straightforward to check that these isomorphisms for varying $m$ provide an isomorphism of graded $C^*(BG)$-${\FF}G$-bimodules.
\end{proof}
In degree zero the cup product is given by pointwise multiplication.

\begin{remark}\label{rem:dualfree}
The $0$-cochains of $EG$ are given by $C^0(EG)=C^0(G)$ since the set of $0$-simplices of $EG$ is $G$. The homomorphism of right ${\FF}G$-modules ${\FF}G\to C^0(G)$ determined by $1\mapsto e^*$ sends $g$ to $(g^{-1})^*$ and hence is an isomorphism.
\end{remark}

Any free $G$-space $X$ is a principal $G$-bundle and thus classified by a map of simplicial sets $X/G\to BG$ (see e.g. \cite[Chapter~IV]{maysimplicialobjects}). We fix a classifying map for the entire article so that $X$ is the pullback
\begin{equation}\label{eq:pullback}
\xymatrix{X\ar[r]^{\tilde{f}}\ar[d]^{\tilde{q}} & EG\ar[d]^q\\
X/G\ar[r]^f & BG.}
\end{equation}
We will use the fact that any two $G$-equivariant maps $X\to EG$ are $G$-homotopic.

\begin{lemma}\label{lem:cochainsfreeGspace} For a free $G$-space $X$, the cup product induces an isomorphism of $C^*(X/G)$-$C^*(EG)$-dg-bimodules:
\[C^*(X) = C^*(X/G\times_{BG}EG)\cong C^*(X/G)\otimes_{C^*(BG)}C^*(EG).\]
\end{lemma}
\begin{proof}We define a map
\[C^*(X/G)\otimes C^*(EG)\rightarrow C^*(X)\quad \varphi\otimes \psi\mapsto \tilde{q}^*(\varphi)\cup \tilde{f}^*(\psi).\]
If we use the other maps in the pullback to equip $C^*(EG)$ with a left $C^*(BG)$-dg-module structure and $C^*(X/G)$ with a right $C^*(BG)$-dg-module structure, then this map factors through \[C^*(X/G)\otimes_{C^*(BG)}C^*(EG).\]
It remains to show that the induced map is an isomorphism. Since $C^*(EG)\cong C^*(BG)\otimes_{{\FF}}C^0(EG)$ by Lemma~\ref{lem:CSTAREGISCSTARBGTENSORC0EG}, it suffices to prove that the map
\begin{align*} \tilde{q}^*\cup \tilde{f}^* \colon C^*(X/G)\otimes_{{\FF}}C^0(EG)\rightarrow C^*(X)\end{align*}
is an isomorphism. We can evaluate this map on two generic cochains $\varphi \in C^m(X/G)$, $\psi \in C^0(EG)$ to obtain
\begin{align*}
    (\tilde{q}^* \cup \tilde{f}^*)(\varphi\otimes \psi)(\sigma)=\varphi(\tilde{q}(\sigma|_{[0,\ldots,m]}))\cdot \psi(\tilde{f}(\sigma_{[m]})).
\end{align*}
The $m$-simplices of our pullback are given by
\[X_m =\{(\sigma,(g_0,\ldots,g_m))\in (X/G)_m\times EG_m\mid f(\sigma)=q((g_0,\ldots,g_m))\}.\]
Since the group elements $(g_1,\ldots,g_m)=q((g_0,\ldots,g_m))=f(\sigma)$ are already determined by $\sigma$, any $m$-simplex is determined by the pair $(\sigma,g_0)$. But it is also determined by $(\sigma,g_0\ldots g_m)$ given that $f(\sigma)=(g_1,\ldots,g_m)$. Note that $(g_0\ldots g_m)\in EG_0$ is the last vertex of $(g_0,\ldots ,g_m)\in EG_m$. For any simplicial $G$-set $Y$, a simplex in $Y$ is nondegenerate if and only if its image in $Y/G$ is nondegenerate. Thus we obtain a bijection
\[\{\sigma \in X_m \mid \sigma\mbox{ nondeg.}\}\stackrel{\tilde{q}\times \tilde{f}(\__{[m]})}{\longrightarrow}\{\sigma \in (X/G)_m \mid \sigma\mbox{ nondeg.}\}\times (EG)_0. \]
With this identification the map $\tilde{q}^* \cup \tilde{f}^*$ above is the standard map 
\[{\FF}^S\otimes {\FF}^T\rightarrow {\FF}^{S\times T}\quad \varphi\otimes \psi\mapsto ((s,t)\mapsto \varphi(s)\cdot \psi(t)),\]
which happens to be an isomorphism if at least one of the sets $S$ or $T$ is finite. Here this holds since $(EG)_0=G$ is finite.
\end{proof}

\begin{remark} The group $G$ acts on the right on $C^*(EG)$ and trivially on $C^*(X/G)$. The isomorphism from the last lemma preserves the $G$-action and thus is an isomorphism of $C^*(X/G)$-${\FF}G$-bimodules as well. Moreover, note that $C^*(X/G)$ injects into $C^*(X)$ as the $G$-invariant cochains.
\end{remark}

\begin{remark}\label{rem:hopfmodulestructure} Lemma~\ref{lem:cochainsfreeGspace} and its proof holds for any finite group $G$ and field of coefficients. An alternate proof for Lemma~\ref{lem:cochainsfreeGspace} is to show that $C^*(X)$ is a Hopf module over the Hopf algebra $C^0(G)$. Then the Fundamental Theorem of Hopf modules (see \cite[1.9.4]{montgomery1993}) provides the isomorphism
\[C^*(X)\cong C^*(X/G)\otimes C^0(G).
\]
In this terminology, the equivariance of the cup product is expressed as follows. The algebra structure of $C^*(X)$ is compatible with the $C^0(G)$-comodule structure in that $C^*(X)$ is a $C^0(G)$-comodule algebra as in \cite[Definition~4.1.2]{montgomery1993}.
\end{remark}
We will often have to express a generic cochain in $C^*(X)$ as a sum of images $m\cup \psi$ of elementary tensors $m\otimes \psi\in C^n(X/G)\otimes_{{\FF}} C^0(G)$. 
For example, given a cochain in that form, we have to use the explicit isomorphisms above to write its boundary again in this form. Especially we will usually define maps only on such elementary tensors.

\begin{remark} Since the group $G$ acts on $C^n(X/G)\otimes_{{\FF}}C^0(G)$ solely on the right factor, we find that the subquotients are 
\[E^{k,q}_0=\frac{C^q(X).I^{L-k}}{C^q(X).I^{L-k+1}}\cong C^q(X/G)\otimes_{{\FF}} \frac{C^0(G).I^{L-k}}{C^0(G).I^{L-k+1}}.\]
Thus additively, the $E_0$-page is the tensor product of $C^*(X/G)$ and the associated graded ring $\gr_\cup(C^0(G))$ of $C^0(G)$ under the $\cup$-product. The isomorphism is induced by pulling back to $C^*(X)$ and using the $\cup$-product in $C^*(X)$.
\end{remark}

Recall that $G$ is the zero-skeleton of $EG$ and thus $C^0(G)=C^0(EG)$. We will use the following lemma to calculate the $E_1$-page.

\begin{lemma}\label{lem:dfdowninfiltr}
We have for $\psi\in F^kC^0(EG)$ that $d\psi\in F^{k-1}C^1(EG)$.
\end{lemma}
\begin{proof}
For a given $\lambda\in I^k$ we know that $\psi.\lambda\in F^0C^0(EG)$ and thus $\psi.\lambda\in C^0(EG)$ is a constant function and hence a cycle, since
\[d(\psi.\lambda)(g_0,g_1)=(\psi.\lambda)(g_0)-(\psi.\lambda)(g_0g_1)=0.\]
So $0=d(\psi.\lambda)=d(\psi).\lambda$. We have shown that $d\psi$ vanishes if we multiply it by any element in $I^k$. Thus  $d\psi$ lies in $F^{k-1}C^1(EG)$ by Definition~\ref{def:filtr}.
\end{proof}

\begin{corollary}\label{cor:d0incupworld} The $E_1$-page of the spectral sequence obtained from the filtration $\{F^k C^*(X)\}$ is additively naturally isomorphic to $H^*(X/G)\otimes\gr_\cup(C^0(G))$.
\end{corollary}
\begin{proof}
Given a generic class $[m\cup \psi]\in E_0^{k,q}$ with $m\in C^q(X/G),\psi\in C^0(EG)$, we obtain that 
\[d(m\cup \psi)=(dm)\cup \psi+\mbox{terms of lower filtration degree}\]
and thus $d_0([m\cup \psi])=[d(m)\cup \psi]$. Hence the $E_1$-page is additively isomorphic to $H^*(X/G)\otimes\gr_\cup(C^0(G))$. We will establish naturality and prove simultaneously that the isomorphism is independent of the chosen pullback classifying the principal bundle. Recall that $C^0(EG)=C^0(G)$. The isomorphism in bidegree $(k,*)$ factors as

\begin{align*}
H^*(X/G)\otimes \frac{F^k C^0(EG)}{F^{k-1}C^0(EG)} &= H^*(X/G)\otimes H^0\left(\frac{F^k C^*(EG)}{F^{k-1}C^*(EG)}\right) \\
&\stackrel{\tilde{q}^*\otimes \tilde{f}^*}{\to} H^*\left(\frac{F^0 C^*(X)}{F^{-1} C^*(X)}\right)\otimes H^0\left(\frac{F^k C^*(X)}{F^{k-1} C^*(X)}\right) \\
& \stackrel{\cup}{\to} H^*\left(\frac{F^k C^*(X)}{F^{k-1} C^*(X)}\right)=E_1^{k,*}.
\end{align*}
The first equality sign holds since the differential
\[d\colon \frac{F^k C^0(EG)}{F^{k-1}C^0(EG)}\rightarrow \frac{F^k C^1(EG)}{F^{k-1}C^1(EG)}\]
 is zero by Lemma~\ref{lem:dfdowninfiltr}. The following map is induced by the natural inclusion of $C^*(X/G)$ into $C^*(X)$ as the $G$-invariant cochains which identifies with $F^0C^*(X)$ and the $G$-map $\tilde{f}\colon X\to EG$ from \eqref{eq:pullback}. Any two representatives for the same $G$-homotopy class of $X\rightarrow EG$ induce equivariantly chain homotopic maps 
$C^*(EG)\rightarrow C^*(X)$ and thus also chain homotopic maps $\frac{F^k C^*(EG)}{F^{k-1}C^*(EG)}\to \frac{F^k C^*(X)}{F^{k-1} C^*(X)}$. This shows simultaneously that $\tilde{q}^*\otimes \tilde{f}^*$ is independent of the chosen pullback classifying the principal bundle and that it is natural. The last map is induced by the cup product and thus natural as well. Hence the composite $H^*(X/G)\otimes \frac{F^k C^0(EG)}{F^{k-1}C^0(EG)}\cong E_1^{k,*}$ is natural.
\end{proof}

To calculate the $E_1$-page multiplicatively, we will need to change the order of a cup product up to a $\cup_1$-product.
\begin{lemma}\label{lem:cupgradcomm} For two cochains $m\cup \psi$, $m'\cup \psi'$ in $C^*(X)$ with $m,m'\in C^*(X/G)$ and $\psi,\psi'\in C^0(EG)$, we have that:
\[m\cup \psi \cup m'\cup \psi' = m\cup m'\cup \psi\cup \psi'- m\cup((d\psi)\cup_1 m')\cup \psi'.\]
\end{lemma}
\begin{proof}
The definition of the differential in the $\Hom$-complex and the defining property of $\cup_1$-products yield as in Section~\ref{sec:conv}:
\[
 m'\cup \psi - \psi\cup m' =(d_{\Hom}(\cup_{1}))(\psi\otimes m') =  d(\psi\cup_{1} m') + (d\psi)\cup_{1} m' + \psi\cup_{1} dm'.
\]
This simplifies to $(d\psi)\cup_{1}m'$ by Remark~\ref{rem:propofcup1} as $\psi$ is a $0$-cochain. Hence $\psi\cup m'=m'\cup \psi - (d\psi)\cup_{1}m'$ and the desired equation follows.
\end{proof}

\begin{corollary}\label{cor:multone1} The isomorphisms 
\begin{align*}E_0^{*,*}&\cong C^*(X/G) \otimes \gr_\cup(C^0(G)),\\
E_1^{*,*}&\cong H^*(X/G) \otimes \gr_\cup(C^0(G))\end{align*}
are multiplicative.
\end{corollary}
\begin{proof}
We already know that these isomorphisms hold additively. It remains to show that the multiplication is really just the tensor product of the two multiplications.
Let two classes in $c\in E^{k,q}_0$ and $c'\in E^{k',q'}_0$ be given. We may restrict to the case where $c,c'$ are elementary tensors.
Choose representatives $m\cup \psi$ of $c$ and $m'\cup \psi'$ of $c'$ with $\psi\in F^kC^0(G),\psi' \in F^{k'}C^0(G)$ and $m\in C^q(X/G), m'\in C^{q'}(X/G)$. By the last lemma, we have:
\begin{align*}m\cup \psi\cup m'\cup \psi'&=m\cup m'\cup \psi \cup \psi' -m\cup ((d\psi)\cup_1 m')\cup \psi'.\\
\end{align*}
The second summand lives in  
$F^{k+k'-1}C^*$ by Lemma~\ref{lem:dfdowninfiltr} and the compatibility of the filtration with $\cup$ and $\cup_1$-products from \cref{lem:filtrismult}. Thus we have shown that
\[[m\cup \psi]\cup[m'\cup \psi'] =[(m\cup m')\cup (\psi\cup \psi')]\]
in $E_0^{k+k',q+q'}$ and hence we have an isomorphism of bigraded differential algebras 
\[E_0^{*,*}\cong C^*(X/G) \otimes \gr_\cup(C^0(G)).\]
The differential on the right-hand side is by Corollary~\ref{cor:d0incupworld} given by the differential on the left factor. Taking homology yields an isomorphism of bigraded algebras
\[E_1^{*,*}\cong H^*(X/G)\otimes \gr_\cup(C^0(G)).\qedhere
\]
\end{proof}

\begin{remark}
To establish the multiplicative structure of the spectral sequence, we relied on Hill's theorem implying that the two filtrations of $C^*$ in Definition~\ref{def:filtr} agree. For general groups $G$ and coefficient fields this does not hold. In light of the homological spectral sequence induced by filtration with powers of the augmentation ideal for general group rings in \cite{papadimasuciu2010spectralsequence}, it could be interesting to see if multiplicativity holds more generally. 
\end{remark}

\section{Review of Jennings bases and the structure of \texorpdfstring{$C^0(G)$}{C0(G)}}
\label{sec:jennings}
In this section we will calculate the associated graded ring $\gr_\cup(C^0(G))$. The computation relies on Jennings bases which we recall from \cite{jennings1941structure}. As mentioned in Remark~\ref{rem:hilljennigsarbf} all results of Jennings cited in this section also work for arbitrary fields of characteristic $p$.

Readers mainly interested in \cref{mthm} may skip this section if they are willing to take for granted that for $G=(\IZ/p)^a$ the associated graded ring of $C^0(G)$ is a truncated polynomial ring
\[\gr_\cup(C^0(G))\cong \frac{\FF[y_1,\ldots,y_n]}{(y_1^p,\ldots,y_n^p)},
\]
where $y_i\colon (\IZ/p)^a\to \FF$ projects onto the $i$-th coordinate $\IZ/p\subset \FF$. 

Recall that $L$ denotes the largest integer such that the $L$-th power of the augmentation ideal $I$ is nonzero.
\begin{definition}[{\cite[Theorem~2.2]{jennings1941structure}}] The group $G$ has a filtration by characteristic subgroups
\[G=G_1\supset G_2\supset \ldots \supset G_{L+1}=1\]
given by 
\[G_i\coloneqq \{g\in G\mid g-1 \in I^i\}.\]
\end{definition}
 The subgroup $G_{L+1}$ is trivial by definition of $L$ and the filtration. It can happen that $G_i$ is trivial for some $i<L+1$ as in the following example.
\begin{example}
\begin{enumerate}[leftmargin=*]
\item
For $G=(\ZZ/p)^a$, we have $G_1=G$ and $G_2=1$.
\item
For $p=2$ and $G=\IZ/8=\langle t\mid t^8 \rangle$, we have $L=7$ and the Jennings filtration is
$G_1 = G,\quad G_2=\langle t^2\rangle,\quad  G_3 = G_4 = \langle t^4 \rangle,\quad G_5 = \langle 1 \rangle$. Thus $G_5=1$ although $I^5 \neq 0$; i.e. $I^5$ still contains nontrivial elements, just none of the form $g-1$. This is an example with $G_i=G_{i+1}\neq 1$ for some $i$, thus $G_{i+1}$ is not determined by $G_i\subset G$.
\end{enumerate}
\end{example}

The Jennings filtration satisfies the following basic property.
\begin{lemma}[{\cite[Theorem~2.3]{jennings1941structure}}]\label{lem:commGi} For the Jennings filtration $\{G_i\}$,  the commutator $[G_i,G_j]$ is contained in $G_{i+j}$ and all $p$-th powers of elements in $G_i$ are contained in $G_{ip}$. Consequently each quotient $G_i/G_{i+1}$ is an elementary abelian $p$-group.
\end{lemma}
In the following definition we identify elementary abelian $p$-groups with finite dimensional $\FF_p$-vector spaces.
\begin{definition}
For $i\ge 1$, select ${\FF_p}$-bases of $G_i/G_{i+1}$, choose a representing group element of each basis vector and take the union of these ordered sets. Call the resulting group elements $f_1,\ldots, f_a$ and let $\alpha\colon\{1,\ldots, a\}\rightarrow \IN$ be the function given by $f_j \in G_{\alpha(j)}\setminus G_{\alpha(j+1)}$. Such a list of group elements, ordered by increasing degrees, is called a \emph{Jennings basis} of $G$.
\end{definition}
The following result follows immediately from the definition.
\begin{proposition}{\cite[p.~178]{jennings1941structure}} Each group element can be uniquely written in the form 
\[g=f_1^{x_1}\ldots f_a^{x_a}\]
with $0\le x_i<p$.
\end{proposition}

Thus if $G$ has $p^a$ elements, then the Jennings basis consists of $a$ group elements. 

We fix one Jennings basis $f_1,\ldots,f_a$ and work with this Jennings basis for the rest of the paper.

\begin{remark}\label{rem:multgrpelementsnormalform} For two group elements written in the Jennings basis, it is hard in general to compute their product in the Jennings basis. This is already the case when multiplying a group element $g=f_1^{x_1}\ldots f_a^{x_a}$ with some $f_i$. Lemma~\ref{lem:commGi} implies that $[f_i]$ is central in $G/G_{\alpha(i)+1}$ and thus 
\[f_i \cdot f_1^{x_1}\ldots f_a^{x_a}= f_1^{x'_1}\ldots f_{a}^{x'_{a}}\]
with $x_i'=x_i+1$ mod $p$ and for $j\neq i$ with $\alpha(j)\le\alpha(i)$ we have $x_j'=x_j$. If $\alpha(j)>\alpha(i)$, we have no control over the new exponents $x_j'$.

The analogous statement holds if we multiply from the right.
\end{remark}
The Jennings basis yields a basis for the powers of the augmentation ideal.
\begin{theorem}{\cite[Theorem~3.2]{jennings1941structure}}\label{thm:basisforIk} An $\FF$-basis for $I^k$ is given by all expressions of the form
\[(f_1-1)^{x_1}\ldots (f_a-1)^{x_a}\]
with $0\le x_i<p$ and $\sum_{i=1}^a\alpha(i)x_i\ge k$. In particular, the products with $\sum_{i=1}^a\alpha(i)x_i=k$ form a basis for $I^k/I^{k+1}$. Moreover, we have  $L=(p-1)\sum_{i=1}^a\alpha(i)$.
\end{theorem}

The choice of a Jennings basis allows us to define specific cochains in  $C^0(G)=\FF^G$, the $\FF$-valued functions on $G$, as follows.

\begin{definition}
Define elements  $y_1,\ldots, y_a$ in ${\FF}^G$ via:
\[y_i(f_1^{x_1}\ldots f_a^{x_a}) \coloneqq x_i\in \{0,\ldots,p-1\} \subset {\FF}.\]
\end{definition}

We will use these cochains to calculate the ring $C^0(G)$ and its associated graded ring.

\begin{remark}
It was helpful for us to use the following analogies to more classical situations:
\begin{table}[H]
    \centering
 \begin{tabular}{|c|c|}
 \hline
  $0$-cochains on $EG$ & polynomials in $y_1,\ldots,y_a$\\ \hline
  $\alpha(i)$ & the degree of the i-th variable $y_i$\\\hline
  choosing a different Jennings basis & \begin{tabular}{@{}c@{}}applying a graded automorphism \\ of the polynomial ring\end{tabular} \\\hline
  group ring elements & differential operators\\\hline
  $\lambda \in I^k$ & differential operators of degree $\ge k$\\\hline
  $F^kC^*$& polynomials of degree $\le k$\\\hline
$F^kC^*$& \begin{tabular}{@{}c@{}}all polynomials on which all differential \\ operators of degree $\ge k+1$ vanish\end{tabular}\\\hline
  \end{tabular}  
\end{table} 
\end{remark}

The key difference to the multinomial case is that here, the elements $y_i$ have degree $\alpha(i)$. If $G$ is an elementary abelian $p$-group all $\alpha(i)$'s are one, and in this case the analogy is more striking.

 Recall that the cup product on $C^0(G)$ is given by pointwise multiplication.

\begin{lemma} The ring $C^0(G)$ (with pointwise multiplication) is isomorphic to 
\[\frac{{\FF}[y_1\ldots,y_a]}{0=y_1^p-y_1=\ldots =y_a^p-y_a}.\]
\end{lemma}
\begin{proof}
The choice of a Jennings basis provides an identification of the underlying set of $G$ with $\{0,\ldots,p-1\}^a=(\FF_p)^a$ via 
\[f_1^{x_1}\ldots f_a^{x_a} \mapsto (x_1,\ldots,x_a).\]
 All maps of sets $(\FF_p)^a\rightarrow {\FF}$ form a ring under pointwise multiplication and addition. It is a fact that the map that associates to a polynomial in ${\FF}[x_1,\ldots,x_n]$ its associated polynomial function surjects onto the set of maps $(\FF_p)^a\rightarrow {\FF}$ with kernel the ideal $(x_1^p-x_1,\ldots,x_a^p-x_a)$. A short argument goes as follows. The characteristic function $1_x$ that sends $x\in(\mathbb{F}_p)^a$ to one and all other points to zero can be expressed as the polynomial function $1_x(z)=\prod_{i=1}^a (1-(z_i-x_i)^{p-1})$ and thus an arbitrary function $f\colon (\FF_p)^a\rightarrow {\FF}$ as $f=\sum_{x\in (\FF_p)^a} f(x) 1_x$. The kernel of the map sending a polynomial in $\FF[x_1,\ldots,x_n]$ to its associated polynomial function contains the ideal $(x_1^p-x_1,\ldots,x_a^p-x_a)$ because $z^p=z$ in $\FF$ for all $z\in \mathbb{F}_p$. The induced surjection from ${\FF[x_1,\ldots,x_a]}/{(x_1^p-x_1,\ldots,x_a^p-x_a)}$ to the ring of functions $(\mathbb{F}_p)^a\to \FF$ is an isomorphism since domain and codomain have the same dimension as $\FF$-vector spaces.
 
 Finally, under the identification $\FF_p^a\cong G$ from above, the polynomial $x_i$ corresponds to $y_i\in {\FF}^G$.
\end{proof}

We will calculate the associated graded $\gr_\cup(C^0(G))$ in \cref{prop:assocgradedofc0} based on the following lemmas. The right ${\FF}G$-module $C^0(G)$ is free of rank one as mentioned in \cref{rem:dualfree}. We have a canonical generator of ${\FF}^G$ as a right ${\FF}G$-module, the dual of the neutral element. Through this generator, the elements $y_i\in C^0(G)$ can be expressed as follows. 

\begin{lemma}\label{lem:yiasgroupringlements} We have 
\[y_i = e^*.((f_a^{-1}-1)^{p-1}\ldots  (f_{i+1}^{-1}-1)^{p-1}f_i^{-1}(1-f_{i}^{-1})^{p-2}(f_{i-1}^{-1}-1)^{p-1}\ldots(f_1^{-1}-1)^{p-1}),\]
where $e^*\in C^0(G)$ denotes the dual of the neutral element.
\end{lemma}
\begin{proof}
We evaluate both sides at an arbitrary element $h\in G$. Write $h=f_1^{x_1}\ldots f_a^{x_a}$. Then $y_i(h)=x_i$ by definition.

For the right-hand side, we first multiply $e^*$ with a generic group ring element $\sum_g\mu_gg$:
\[e^*.(\sum_g\mu_gg)(h)= \sum_g\mu_{g} e^*(gh) = \mu_{h^{-1}}.\]
So to check that the right-hand side evaluated at $h$ is $x_i$, we have to expand the product appearing in the right-hand side and check that the coefficient of $h^{-1}=(f_a)^{-x_a}\ldots (f_1)^{-x_1}$ is $x_i$.
We have: 
\begin{align*}(f_k^{-1}-1)^ {p-1}& = 1+f_k^{-1}+f_k^{-2}+\ldots f_{k}^{-p+1},\\
f_i^{-1}(1-f_{i}^{-1})^{p-2} &= f_i^{-1}+2f_i^{-2}+3f_i^{-3} + \ldots +(p-1)f_i^{-p+1}.
\end{align*}
The first equation holds since $\binom{p-1}{j}=(-1)^j$ in $\FF$. The second equation can be obtained by expanding the left-hand side and simplifying the binomial coefficients as follows: 
\begin{align*}
\binom{p-2}{j}&=\frac{(p-2)!}{j!(p-2-j)!}\\
&=\frac{-(p-1)!}{j! \cdot (j+2)(j+3)\ldots (p-1)\cdot(-1)^{p-2-j}}\\
&= (-1)^{p-1-j} (j+1).
\end{align*}
Thus in the expansion of the original product appearing in the right-hand side, each group element occurs exactly once and the coefficient of $h^{-1}=(f_a)^{-x_a}\ldots (f_1)^{-x_1}$ is $x_i$.
\end{proof}
We use Lemma~\ref{lem:yiasgroupringlements} to calculate the filtration degrees of the $y_i$'s.
\begin{lemma}\label{lem:yifiltrationdegree} We have $y_i\in F^{\alpha(i)}C^0(G)$ and $y_i\notin F^{\alpha(i)-1}C^0(G)$.
\end{lemma}
\begin{proof} 
By Lemma~\ref{lem:yiasgroupringlements}, we have:
\[y_i = e^*.((f_a^{-1}-1)^{p-1}\ldots  (f_{i+1}^{-1}-1)^{p-1}f_i^{-1}(1-f_{i}^{-1})^{p-2}(f_{i-1}^{-1}-1)^{p-1}\ldots(f_1^{-1}-1)^{p-1}) .\]
Furthermore, $f_k^{-1}-1 = f_k^{-1}(1-f_k)$ lies in $I^{\alpha(k)}$ and thus the entire product lies in $I^{\sum_j (p-1)\alpha(j)-\alpha(i)}$. 
Since $L=(p-1)\sum_{j=1}^a\alpha(j)$ by Theorem~\ref{thm:basisforIk}
we conclude that $y_i$ lies in $C^0(G).I^{L-\alpha(i)}$ which equals $F^{\alpha(i)}C^0(G)$ by Definition~\ref{def:filtr}. So we have shown the first statement.
For the second part, let us compute
\[y_i.(f_i-1)(e) =((y_i.f_i)-y_i)(e)= y_i(f_i)-y_i(e)=1-0=1,\]
where $e\in G$ is the neutral element. Thus the element $(f_i-1)\in I^{\alpha(i)}$ satisfies
\[y_i.(f_i-1)\neq 0.\]
It follows that $y_i\notin F^{\alpha(i)-1}C^0(G)$ since $\lambda=f_i-1$ is an element of $I^{\alpha(i)}$ satisfying $y_i.\lambda\neq 0$. 
\end{proof}

The following result provides an easy criterion to check if an element of $C^0(G)$ lies in $F^{L-1}C^0(G)$.
\begin{lemma}\label{lem:c0(G)I} The set $C^0(G).I$ consists of all $f\colon G\to \FF$ with $\sum_g f(g)=0$.
\end{lemma}
\begin{proof}
Since $e^*$ is the characteristic function of the neutral element, we have for $f\in C^0(G)$ that $f=e^*.(\sum_g f(g^{-1})g)$. The composite
\[C^0(G)\cong \FF G\to \FF
\]
consisting of the isomorphism $e^*\mapsto 1$ of \cref{rem:dualfree} followed by the augmentation map sends $f$ to $\sum_g f(g^{-1})=\sum_g f(g)$. This sum vanishes if and only if $f$ is in the kernel of the composite and this kernel is $C^0(G).I$ since $I$ is the augmentation ideal.
\end{proof}

We exhibit a generator of the one-dimensional vector space $F^L C^0(G)/ F^{L-1} C^0(G)$ as a product of $y_i$'s using the previous result.
\begin{lemma}\label{lem:prodyineqzero}$y_1^{p-1}\cup\ldots\cup y_a^{p-1}\in F^{L}C^0(G) \setminus F^{L-1}C^0(G)$.
\end{lemma}
\begin{proof}
Since $F^LC^0(G)$ is all of $C^0(G)$, we just have to show that $y_1^{p-1}\cup\ldots\cup y_a^{p-1}$ is not in $F^{L-1}C^0(G)=C^0(G).I$. This cochain sends a generic group element
$f_1^{x_1}\ldots f_a^{x_a}$ to zero, if at least one $x_i$ is zero and to one otherwise. To check whether a function $f\in C^0(G)$ lies in $C^0(G).I$, we have to check whether $\sum_{g\in G}f(g)=0\in {\FF}$ by Lemma~\ref{lem:c0(G)I}. This is not the case here, since that sum is $(p-1)^a=(-1)^a$. 
\end{proof}

We are ready to calculate $\gr_\cup(C^0(G))$.

\begin{proposition}\label{prop:assocgradedofc0} The associated graded ring $\gr_\cup(C^0(G))$ of $(C^0(G),\cup)$ is given by the truncated polynomial ring
\[\frac{{\FF}[y_1,\ldots,y_a]}{0=y_1^p = \ldots =y_a^p},\]
where each $y_i$ lies in filtration degree $\alpha(i)$.
\end{proposition}
\begin{proof}
The relations $y_i^p=y_i$ in $C^0(G)$ together with the statement that each $y_i$ lies in positive filtration degree by Lemma~\ref{lem:yifiltrationdegree} imply that the relations $y_i^p=0$ hold in the associated graded ring. So we obtain a map of graded rings
\[ \Phi\colon \frac{\FF[y_1,\ldots,y_a]}{0=y_1^p=\ldots =y_a^p}\rightarrow \gr_\cup(C^0(G)).\]
Since source and target have the same vector space dimension, it suffices to establish injectivity. Consider a homogeneous element
\[\sum_x \mu_{x} y_1^{x_1} \ldots y_a^{x_a}\]
in the kernel of $\Phi$. For any fixed multiindex $k$, we will show that $\mu_k=0$. Each occurring multiindex $x\in \{0,\ldots,p-1\}^a$ satisfies $\sum_i x_i\alpha(i)=\sum_i k_i\alpha(i)$. Hence for $x\neq k$ there must exist an index $i$ with $x_i>k_i$ and therefore,
\[\left(y_1^{p-1-k_1}\ldots  y_a^{p-1-k_a}\right)  \left(\mu_{x} y_1^{x_1} \ldots  y_a^{x_a}\right)=\mu_x y_1^{p-1+x_1-k_1}\ldots y_a^{p-1+x_a-k_a}\]
is zero in the domain of $\Phi$ by the relation $y_i^p=0$.
For $x=k$, this product equals  
\[\mu_k y_1^{p-1}\ldots y_a^{p-1} .\]
We deduce that \[\left(y_1^{p-1-k_1}\ldots  y_a^{p-1-k_a}\right) \sum_x \left(\mu_{x} y_1^{x_1} \ldots  y_a^{x_a}\right)=\mu_k y_1^{p-1}\ldots y_a^{p-1}.\]
Applying $\Phi$ to the left-hand side yields zero by assumption on $\sum_x \left(\mu_{x} y_1^{x_1} \ldots  y_a^{x_a}\right)$. Since $\Phi(y_1^{p-1}\ldots y_a^{p-1})\neq 0$ by Lemma~\ref{lem:prodyineqzero}, it follows that $\mu_k=0$. This completes the proof.
\end{proof}

 In general the graded ring induced by the $\cup$ product $\gr_\cup(C^0(G))$ and the graded ring induced by the group ring multiplication $\gr(\FF G,\cdot)$ are not isomorphic. For example for $G=\IZ/4$ we have that $\gr_\cup(C^0(G))=\FF_2[y_1,y_2]/(y_1^2,y_2^2)$ with $|y_1|=1, |y_2|=2$, while $\gr(\FF G,\cdot)$ is $\FF_2[t]/t^4$  with $|t|=1$.

\section{The differential on the first page}\label{sec:d1}
The goal of this section is to determine the differential $d_1:E_1^{*,*}\rightarrow E_1^{*-1,*+1}$ for the spectral sequence 
\[E_1^{*,*}\cong H^*(X/G)\otimes  \gr_\cup(C^0(G))
\]
arising from a free $G$-space $X$. This is the last piece for Theorem~\ref{mthm:1} which we will establish in the end of this section. The differential $d_1$ of an element $m\otimes 1\in E_1^{0,*}$ is zero for degree reasons. Thus by the Leibniz rule and the multiplicative structure of the $E_1$-page, the differential $d_1$ is fully determined by $d_1(1 \otimes y_i)$ for $y_i\in F^{\alpha(i)}C^0(G)$. By the naturality established in Corollary~\ref{cor:d0incupworld}, the map $X\rightarrow EG$ arising from a classifying map $f\colon X/G\to BG$ induces a map of spectral sequences which on the $E_1$-page is 
\[f^*\otimes \id\colon H^*(BG)\otimes \gr_\cup(C^0(G))\rightarrow H^*(X/G)\otimes \gr_\cup(C^0(G)).\]
Thus to compute $d_1(1\otimes y_i)$ in the target, we can compute it in the source and apply this map. It seems to be quite challenging to do this for general $p$-groups and all $y_i$. We will first calculate $d_1(1\otimes y_i)$ for $y_i$ with $\alpha(i)=1$ and then outline a strategy to compute $d_1$ in general in Remark~\ref{rem:d1yjingeneral}.

For $\alpha(i)=1$, the differential $d_1(1 \otimes y_i)$ lands in $E_1^{1,0}\cong H^1(BG)\otimes F^0(C^0(G))\cong H^1(BG)$. Recall that  $H^1(BG;{\FF})$ is the abelian group of group homomorphisms $\Hom(G,\FF)$. We continue writing $f_1,\ldots,f_a$ for a Jennings basis for $G$ and recall that for $G=(\ZZ/p)^a$ the canonical $a$ generators form a Jennings basis.
\begin{definition}\label{def:ai}
If $\alpha(i)$ is one, let $a_i \colon G\rightarrow {\FF}$ denote the map \[(f_1^{x_1}\ldots f_a^{x_a})\mapsto -x_i.\] 
\end{definition}
Since  $\alpha(i)=1$, this map is a group homomorphism by Remark~\ref{rem:multgrpelementsnormalform}. As a function, $a_i$ is just $-y_i$. We denote it differently, since we consider $y_i\in C^0(EG)=\FF^G$ and $a_i\in C^1(BG)= \{f\in \FF^G\mid f(e)=0\}$. Especially $a_i$ represents an element in $E_1^{0,1}$ and $y_i$ represents an element in $E_1^{1,0}$. Nevertheless, we will use that as a function $y_i$ agrees with $-a_i$ and thus is a group homomorphism in the following proof.

\begin{lemma} \label{lem:dyiisai} We have on the $E_1$-page for $EG$ and for any $i$ with $\alpha(i)=1$:
\[d_1(1\otimes y_i)=a_i\in E_1^{0,1}\cong H^1(BG;\FF)=\Hom(G,{\FF}).\]
Moreover, the $a_i$'s for $\alpha(i)=1$ form a basis of $\Hom(G,{\FF})=H^1(BG;{\FF})$.
\end{lemma}
\begin{proof}
Under the identification with $E_1^{1,0}$, the element $1\otimes y_i$ corresponds to $[y_i]\in H^0\left(\frac{F^{1}C^*(EG)}{F^{0}C^*(EG)}\right)$. By definition of $C^*(EG)$, we have:
\begin{align*}dy_i(g_0,g_1)&=-(-1)^{|y_i|}y_i(d_0(g_0,g_1)-d_1(g_0,g_1))\\
    &=-y_i(g_0g_1)+y_i(g_0)\\
    &=-y_i(g_0)-y_i(g_1)+y_i(g_0)\\
    &=-y_i(g_1)
\end{align*}
Thus $dy_i\in C^1(EG)$ is $a_i\in C^1(BG)\subset C^1(EG)$. It follows that $d_1(1\otimes y_i)=a_i$.
By \cite[Theorem~5.5]{jennings1941structure}, we know that the group $G_2$ in the Jennings filtration of $G$ is spanned by the commutator $[G,G]$ and all $p$-th powers of elements of $G$. Thus any group homomorphism $G\rightarrow {\FF}$ sends $G_2$ to zero. Therefore, the $a_i$'s given above form a basis of $\Hom(G,{\FF})=H^1(BG)$. 
\end{proof}

In the special case of $G=(\ZZ/p)^a$, all $\alpha(i)$'s are one. Hence we have fully described the differential $d_1$ in this case.

\begin{corollary}\label{cor:identifyyiuptounits}
For $G=(\IZ/p)^a$
the differential $d_1$ sends a generic class $[m]\cup [y_1]^{x_1}\cup\ldots\cup [y_a]^{x_a}$ to
\[\sum_{i=1}^a (-1)^{|m|} x_i [m]\cup [a_i] \cup [y_1^{x_1}\ldots y_i^{x_i-1}\ldots y_a^{x_a}]. \]
\end{corollary}
\begin{proof}
Note that the differential $d_1$ sends $[m]\in E_1^{0,*}$ to zero, since its target $E_1^{-1,*+1}$ is zero. Now apply the graded Leibniz rule, Lemma~\ref{lem:dyiisai} and use that we already know the multiplicative structure on the $E_1$-page by Corollary~\ref{cor:multone1}.
\end{proof}

Our strategy to calculate $d_1$ in general involves the action of $\gr(\FF G,\cdot)$ on the $E_1$-page from \cref{rem:multmaps} that we recall here.

\begin{remark}\label{rem:g-1isderivation} Let $g\in G$ with $[g-1]\in I^l/I^{l+1}$ be given. Right multiplication with $(g-1)$ defines maps 
\[.(g-1)\colon F^{k}C^*(X) \rightarrow F^{k-l}C^*(X)\]
inducing 
\[.(g-1)\colon E_r^{k,*}\rightarrow E_r^{k-l,*}\]
which depend only on the class of $[g-1]\in I^l/I^{l+1}$. These maps are derivations by \cref{lem:filtrismult}; the summand $x.\lambda_g\cup y.\lambda_g$ can be ignored, as it lives further down in the filtration. 
Since these maps are derivations and the identification $E_1^{*,*}\cong H^*(X/G)\otimes \gr_\cup(C^0(G))$ from Corollary~\ref{cor:multone1} is given by the $\cup$-product, these maps solely act on the right factor.
\end{remark}

We are particularly interested in the case of $g=f_j$ with $l=\alpha(j)$. To calculate this action on the $E_1$-page, it again suffices to look at the $y_i$'s. 

\begin{lemma}\label{lem:fi-1onyi} If $\alpha(j)=1$, then
\[y_j.(f_i-1) = \begin{cases} 1 &\text{ if }  i=j\\ 0 &\text{ if }  i\neq j.\end{cases}\]
\end{lemma}
\begin{proof} We evaluate in a generic group element $h$. Using again that $y_j$ is a group homomorphism when $\alpha(j)$ is one, we obtain:
\[y_j.(f_i-1)(h) = y_j(f_ih)-y_j(h) = y_j(f_i) =\begin{cases} 1 & \text{ if } i=j\\ 0 & \text{ if } i\neq j.\end{cases} \qedhere
\]
\end{proof}

The following lemma is crucial for our strategy to calculate $d_1$ in general.

\begin{lemma}\label{lem:sthisinjective} Let $J$ be the set of all indices $i$ in $1,\ldots, a$ with $\alpha(i)=1$ and let $k\ge 0$. For the spectral sequence arising from any free $\FF G$-cochain complex $C^*$, the map 
\[E_1^{k+1,*}\rightarrow \bigoplus_{i\in J} E_1^{k,*},\quad [x]\mapsto ([x].(f_i-1))_{i\in J} \]
is injective.
\end{lemma}
\begin{proof} Under the identification of Corollary~\ref{cor:e1byrightaction}, the map above is given as the identity on $H^*(C^*/(C^*.I))$ tensored with  
\[\frac{I^{L-k-1}}{I^{L-k}}\rightarrow  \bigoplus_J \frac{I^{L-k}}{I^{L-k+1}},\quad [x] \mapsto ([x.(f_i-1)])_{i\in J}.\]
Since tensoring with a vector space is exact, it suffices to show injectivity of the second map. 
Now suppose that $x.(f_i-1)\in I^{L-k+1}$ for all $i\in J$. Since the elements $(f_i-1)_{i\in J}$ form a basis of the ${\FF}$-vector space $I/I^2$ by Theorem~\ref{thm:basisforIk} and $x.\lambda \in I^{L-k+1}$ for all $\lambda \in I^2$, it follows that $x.\lambda \in I^{L-k+1}$ for all $\lambda \in I$. By Remark~\ref{rem:filtrationg-1}, this is equivalent to $x\in I^{L-k}$, i.e. $[x]=0$.
\end{proof}
 
\begin{remark}\label{rem:d1yjingeneral} We computed $d_1(y_j)$ for $\alpha(j)=1$ in Lemma~\ref{lem:dyiisai}. We can calculate $d_1(y_j)$ for $j\ge 2$ by induction on $\alpha(j)$ as follows. First compute $y_j.(f_i-1)\in C^0(EG)$ for all $i$ with $\alpha(i)=1$. Since these live in a lower filtration degree, these elements can be written as a polynomial in the $y_k$'s with $\alpha(k)<\alpha(j)$. By induction hypothesis, we know $d_1(y_j.(f_i-1))=d_1(y_j).(f_i-1)$ for all $i$ as above. But this uniquely determines $d_1(y_j)$ by Lemma~\ref{lem:sthisinjective}.
\end{remark}

We illustrate this strategy in an example.
\begin{example} Consider $G=\IZ/4$ and thus $p=2$. Here a Jennings basis is given by $f_1=1,f_2=2$. We determine $d_1$ for the spectral sequence obtained by filtering $C^*(EG)$. The $E_1$-page consists of the four columns 
\begin{align*}
E_1^{0,*}&=H^*(BG)\otimes 1,\\
E_1^{1,*}&=H^*(BG)\otimes y_1, \\
E_1^{2,*}&=H^*(BG)\otimes y_2, \\
E_1^{3,*}&=H^*(BG)\otimes y_1 y_2.
\end{align*}
By multiplicativity, it suffices to compute $d_1(y_1)$ and $d_1(y_2)$. We know that $d_1(y_1)=a_1$ by Lemma~\ref{lem:dyiisai}.
To determine $d_1(y_2)$, we follow the method described above. Therefore, we compute first:
\[y_2.(f_1-1)(f_1^{x_1}f_2^{x_2}) =y_2(f_1f_1^{x_1}f_2^{x_2})-y_2(f_1^{x_1}f_2^{x_2}). \]
So now we have to write $f_1f_1^{x_1}f_2^{x_2}$ again in the Jennings basis and check whether the exponent of $f_2$ has changed. This happens exactly if $x_1$ equals one. Thus $y_2.(f_1-1)$ equals $y_1$. Hence
\[d_1(y_2).(f_1-1)=d_1(y_2.(f_1-1))=d_1(y_1)=a_1.\] 
Recall that multiplication with $(f_1-1)$ is a derivation by Remark~\ref{rem:g-1isderivation}. Using Lemma~\ref{lem:fi-1onyi} and that $a_1.(f_1-1)\in E_1^{-1,1}=0$, we obtain
\[ (a_1\otimes y_1).(f_1-1)=(a_1.(f_1-1))\otimes y_1+a_1\otimes (y_1.(f_1-1))=0+a_1.\]
By Lemma~\ref{lem:sthisinjective}, we conclude that $d_1(y_2) = a_1\otimes y_1$ . So we have completely determined the structure of the $E_1$-page for this group.
\end{example}

We collect the established results to prove Theorem~\ref{mthm:1}.
 \begin{theorem}\label{thm:main1} Let $I\subset {\FF}G$ be the augmentation ideal for a finite $p$-group $G$ and $L$ the nilpotency index of $I$ minus one.
\begin{enumerate}[leftmargin=*]
\item Filtering the ${\FF}$-cochains of a free, $G$-space $X$ via $\{C^*(X).I^{L-k}\}_k$
gives rise to a natural, multiplicative spectral sequence $\{d_r\colon E_r^{k,q}\rightarrow E_r^{k-r,q+1}\}$ converging to $H^*(X)$ whose $E_1$-term is as a bigraded algebra naturally isomorphic to 
\[E_1^{k,q}\cong H^q(X/G)\otimes \gr_\cup^{k}(C^0(G)),\]
where $\gr_\cup(C^0(G))$ denotes the graded ring associated to filtering ${\FF}^G=C^0(G)$ equipped with the pointwise multiplication by powers of the augmentation ideal $I^{L-k}$.
\item This graded ring is isomorphic to ${\FF}[y_1,\ldots,y_a]/(y_1^p,\ldots,y_a^p)$, where $p^a$ is the order of $G$, and the degree $\alpha(j)$ of the classes $y_j$ correspond to the degrees of the elements of a Jennings basis.
\item For classes $y_j$ with $\alpha(j)=1$, the differential $d_1$ is given by
\[d_1(x\otimes y_j) = (-1)^{|x|} x\cup a_j,\] 
where the $a_j$ form an explicit basis of $H^1(BG;{\FF})$ and $H^*(BG)$ acts on $H^*(X/G)$ through a choice of a classifying map $X/G\to BG$.
\item Furthermore, multiplication with a group ring element of the form $g-1\in I^k$ induces a derivation $E^{*,*}_r\rightarrow E^{*-k,*}_r$ on each page which only depends on $[g-1]\in \frac{I^k}{I^{k+1}}$.
\item These derivations for $[g-1]\in I/ I^2$ and $d_1(y_j)$ when $\alpha(j)$ is one already determine the entire differential on the $E_1$-page.
\end{enumerate}
\end{theorem}
\begin{proof}
\begin{enumerate}[leftmargin=*]
\item Since the filtration is natural, the spectral sequence is natural as well. The spectral sequence is multiplicative by Proposition~\ref{prop:filtismult} and the $E_1$-page is given as above by Corollary~\ref{cor:multone1}. This isomorphism is natural by Corollary~\ref{cor:d0incupworld}. The spectral sequence converges to $H^*(X)$, since the filtration has finite length $L$.
\item This is Proposition~\ref{prop:assocgradedofc0}.
\item This follows from the graded Leibniz rule and  Lemma~\ref{lem:dyiisai} with $a_j$ as in Definition~\ref{def:ai}.
\item This is Remark~\ref{rem:g-1isderivation}.
\item This is demonstrated in Remark~\ref{rem:d1yjingeneral}.\qedhere
\end{enumerate}
\end{proof}
It would be interesting to find explicit formulas and a group theoretic interpretation for $d_1(x\otimes y_j)$ when $\alpha(j)\geq 2$.

For an arbitrary $G$-space $Y$, applying Theorem~\ref{thm:main1} to $X=EG\times Y$ yields the following spectral sequence for the equivariant Borel cohomology $H^*_G(Y)$.
\begin{corollary} For any $G$-space $Y$, there is a natural, multiplicative spectral sequence with $E_1$-page
\[E_1^{k,q}\cong H^q_G(Y)\otimes \gr_\cup^{k}(C^0(G))\]
converging to $H^*(Y)$.
\end{corollary}

The spectral sequence from Theorem~\ref{thm:main1} relates to the Eilenberg-Moore spectral sequence as follows.
\begin{remark}\label{rem:eilenbergmoore} %rem:semifree
%The filtration $\{F^k C^*(EG)\}_k$ has free subquotients, thus $C^*(EG)$ is a semifree $C^*(BG)$-dg-module.
The $C^*(BG)$-dg-module $C^*(EG)$ is semifree (see \cref{rem:semifree} for the Definition) via the filtration $\{F^k C^*(EG)\}_k$.
For $p=2$ and an elementary abelian $2$-group $G$, the spectral sequence agrees with the Gugenheim-May construction \cite{gugenheimmay1974} of the Eilenberg-Moore spectral sequence for the pullback
\[\xymatrix{X\ar[r]\ar[d] & EG\ar[d] \\
X/G \ar[r] & BG.}\]
For elementary abelian $p$-groups $G$ with $p$ odd, the $E_2$-page is not isomorphic to $\Tor_{H^*(BG)}(H^*(X/G),\FF)$. The assumption on \cite[p.~2]{gugenheimmay1974} that the sequence
\[\ldots \to E_1^{k,*}(C^*(EG))\to E_1^{k-1,*}(C^*(EG))\to \ldots \to E_1^{0,*}(C^*(EG))\to H^*(EG)\]
is exact holds if and only if $p=2$.
\end{remark}
Carlsson proved in \cite{carlsson1986} that for $p=2$, the algebraic version of the $p$-toral rank conjecture implies a weak form of the Buchsbaum-Eisenbud-Horrocks conjecture working over a graded polynomial ring in characteristic $2$. We establish the following converse connection from the weak form of the Buchsbaum-Eisenbud-Horrocks conjecture to Carlsson's conjecture, shedding light on the difficulty of the latter.

\begin{remark}\label{rem:buchsbaumeisenbudhorrocks} 
Let $p=2$ and let $G$ be an elementary abelian $2$-group of rank $a$. Thus $H^*(BG)$ is a polynomial ring in $a$ variables. The Buchsbaum-Eisenbud-Horrocks conjecture for a graded module $M$ over $R=H^*(BG)$ of finite total $\FF$-dimension states that the dimension of $\Tor^{R}_i(M,\FF)$ is at least $\binom{a}{i}$. It would be interesting to know if the weaker version $\sum_i \Tor^{R}_i(M,\FF)\geq 2^a$
holds. This holds over fields of characteristic different from $2$ by work of Walker \cite{walker2017}. 

The $\Tor$-terms still have an internal grading. If $M=H^*(X/G)$ for a finite, free $G$-CW complex $X$ with the module structure induced by the ring map $H^*(BG)\rightarrow H^*(X/G)$ the term $\Tor^{R}_i(M,\FF)$ is $E_2^{i,*}$ either by \cref{rem:eilenbergmoore}, or direct identification of the $E_1$-page with $H^*(X/G)\otimes_{H^*(BG)} K$, where $K$ is the Koszul resolution of $H^*(BG)$. Thus the weak version of the Buchsbaum-Eisenbud-Horrocks conjecture implies Carlsson's conjecture for $p=2$ and finite, free $G$-CW complexes $X$ such that the spectral sequence collapses on the $E_2$-page. Nevertheless, the following example is the cellular cochain complex of a free $G$-CW complex with $G=(\ZZ/2)^2=\langle f_1,f_2\rangle$ for which $d_2$ is nonzero: 
\[\FF G\stackrel{\begin{pmatrix}\lambda_1\\\lambda_2\end{pmatrix}}{\xrightarrow{\hspace*{1.5cm}}} \FF G\oplus\FF G\stackrel{\begin{pmatrix}\lambda_1\lambda_2&0\end{pmatrix}}{\xrightarrow{\hspace*{1.5cm}}} \FF G,\]
where $\lambda_1=f_1-1$ and $\lambda_2=f_2-1$.
\end{remark}

\section{Nonrealizable mapping cones}
\label{sec:nrmappingcones}

In this section, let $G$ be an elementary abelian $p$-group $(\IZ/p)^a$ of rank $a$ generated by $f_1,\ldots,f_a$. The group ring $\FF G$ is isomorphic to the truncated polynomial ring
\[\FF G\cong \frac{\FF[\lambda_1,\ldots,\lambda_a]}{(\lambda_1^p,\ldots,\lambda_a^p)}
\]
with $\lambda_i=f_i-1$. Under this identification the augmentation ideal $I$ is $(\lambda_1,\ldots,\lambda_a)$.

The examples of free $\FF G$-chain complexes by Iyengar and Walker arise as a mapping cone built out of a specific Koszul complex $K_*$. We will prove that these mapping cones can not be realized topologically. The general strategy that a given free $\FF G$-cochain complex $D^*$ can not be realized topologically as the cochains $C^*(X;\FF)$ on a free $G$-space $X$ is as follows. First, calculate the spectral sequence algebraically from $D^*$. Secondly, compute the potential $H^1(BG)$-action on $H^*(D^*/D^*.I)\cong H^*(X/G)$. This computes part of the potential ring structure of $H^*(D^*/D^*.I)$ and hence of the product structure of the spectral sequence. Thirdly, find contradictions to the Leibniz rule. The simplest case of our family of nonrealizable cochain complexes is Example~\ref{ex:smallnonrealizable}. There the differential $d_2$ violates the Leibniz rule for the class in $E^{2,2}_2$ in Figure~\ref{figexample}, which would be the product of two classes in $E_2^{1,1}$. Before introducing the specific Koszul complex $K_*$, we recall the definition of Koszul complexes working over $\FF G$. The reader is referred to \cite{brunsherzog1993} for additional background information.
 
 Let $f:M\rightarrow \FF G$ be a homomorphism of $\FF G$-modules.
  The Koszul complex $K_*(f)$ of $f$ is a dga whose underlying graded algebra is the exterior algebra of $M$. The differential is obtained by extending $f\colon K_1(f)\rightarrow K_0(f)$ via the graded Leibniz rule. In the special case where $M=(\FF G)^a$ we will denote its standard basis by $z_1,\ldots,z_a$. Then the map $f$ is determined by the images of these basis elements and we will also speak of the Koszul complex of the sequence $f(z_1),\ldots,f(z_a)$. 
 
In this section, we consider the Koszul complex of the minimal generating sequence $(\lambda_1,\ldots, \lambda_a)$ of $I$ and we denote it just by $K_*$. Moreover, we write $\wedge$ for the product in the Koszul complex.
 The homology of the Koszul complex $K_*$ is an exterior algebra in the $a$ classes $[\lambda_1^{p-1}z_1],\ldots,[\lambda_a^{p-1}z_a]\in H_1(K_*)$ over $\FF$, see \cite[Theorems~2.3.2 and 2.3.11]{brunsherzog1993}.

We work with the following basis for $K_m$. For an $m$-element subset $s\subset \{1,\ldots,a\}$ let $s_i$ be the $i$-th element of $s$ if we order the elements increasingly, i.e.  $s=\{s_1,\ldots,s_m\}$ with $s_1<\ldots<s_m$. Further, let $z_s$ denote the element $z_{s_1}\wedge \ldots \wedge z_{s_m}$ in $K_{m}$. The $m$-th chain module $K_m$ has a basis given by all $z_s$, where $s$ runs through all subsets of $\{1,\ldots,a\}$ of cardinality $m$. The graded Leibniz rule implies that the differential on $K_*$ can be expressed as \[d(z_s)=\sum_{i=1}^m (-1)^i\lambda_{s_i} z_{s\setminus \{s_i\}}.\]

To show that the mapping cones constructed by Iyengar and Walker can not be realized topologically, we will need to work cohomologically and understand the spectral sequence for $K^*$, the dual cochain complex of $K_*$.

For a free $\FF G$-module $V$ with basis $b_1,\ldots,b_n$, we obtain an $\FF G$-basis $b_1^*,\ldots,b_n^*$ for $\Hom_\FF(V,\FF)$ with \[b_i^*(gb_{i'})=\begin{cases} 1 & \text{ if } i=i' \mbox{ and  } g=e\\ 0 &\text{ otherwise}\end{cases}\]
by Remark~\ref{rem:dualfree}. In particular, we obtain basis elements $z_s^*$ of $K^m$ from the basis given by $z_s$ with $|s|=m$ of $K_m$. The differential is $d_{K^*}(f)=-(-1)^{|f|}f\circ d_{K_*}$.

For elements $m,m'$ of an $\FF$-vector space, we will write $m \equiv m'$ if $m=um'$ for some unit $u\in \FF$.

In order to relate the algebraic computation of the spectral sequence to the multiplicative structure, we need to connect the two different ways of expressing the $E_1$-page: 
\[E^{L,q}_1\otimes \frac{I^{L-k}}{I^{L-k+1}}\cong E^{k,q}_1\cong E^{0,q}_1\otimes \gr^k_\cup(C^0(G)).\]
The first one in Corollary~\ref{cor:e1byrightaction} uses just the ${\FF}G$-module structure and is suitable for explicit computations. The second one in Corollary~\ref{cor:multone1} uses multiplicativity and that the cochain complex arises from a $G$-space. The next proposition allows us to pass between these two using the following notation.
\begin{definition}
Let $N\in I^L\subset {\FF}G$ be the norm element $N=\sum_{g\in G} g$.
\end{definition}

Note that $N=\lambda_1^{p-1}\ldots \lambda_a^{p-1}$ and that $E_1^{L,q}\to E_1^{0,q},\quad c\mapsto c.N,$ is an isomorphism by Corollary~\ref{cor:e1byrightaction}.

\begin{proposition}\label{prop:findingyi} Let $G=(\IZ/p)^a$ and let $C^*$ be a cochain complex of free, right ${\FF}G$-modules. Assume that $C^*$ is equivariantly homotopy equivalent to $C^*(X)$ for some free $G$-space $X$.
Then we have on the $E_1$-page for any multiindex $(j_1,\ldots,j_a)\in \{0,\ldots,p-1\}^a$ and any class $c\in E^{L,*}_1$: 
\[c.\lambda_1^{j_1}\ldots \lambda_a^{j_a}\equiv c.N\cup y_1^{p-1-j_1}\cup\ldots \cup y_a^{p-1-j_a}.\]
\end{proposition}
\begin{proof}
We proceed by induction on $\sum_i j_i$. If that sum equals zero, then we have to establish the relation \[c\equiv c.N\cup y_1^{p-1}\cup\ldots \cup y_a^{p-1}.\] The right-hand side is given by applying the two isomorphisms 
\begin{align*}
     E_1^{L,*}\rightarrow E_1^{0,*},&\quad  c\mapsto c.\lambda_1^{p-1}\ldots \lambda_a^{p-1},\\
    E_1^{0,*}\rightarrow E_1^{L,*},&\quad m\mapsto m\cup y_1^{p-1}\ldots y_a^{p-1},
\end{align*} 
from Corollary~\ref{cor:e1byrightaction} and Corollary~\ref{cor:multone1} to $c\in E_1^{L,*}$. Thus it suffices to show that these are up to a unit inverses of each other, and since they are isomorphisms, it suffices to consider one of the compositions. 
Multiplying by $\lambda_i$ is a derivation $E_r^{*,*}\rightarrow E_r^{*-1,*}$ by Remark~\ref{rem:g-1isderivation}. It sends $m\in E_1^{0,*}$ to zero, since the target would live in $E_r^{-1,*}=0$, and we know its values on $y_1,\ldots,y_a$ by Lemma~\ref{lem:fi-1onyi}. Hence 
\[(m\cup y_1^{p-1}\ldots y_a^{p-1}).\lambda_1^{p-1}\ldots \lambda_a^{p-1} =((p-1)!)^am,\]
which establishes the base case of the induction. 

For the induction step, choose an index $l$ with $j_l\ge 1$. Setting $j'_i=j_i$ for $i\neq l$ and $j'_l=j_l-1$, the induction hypothesis yields \[c.\lambda_1^{j'_1}\ldots \lambda_a^{j'_a}\equiv c.N\cup y_1^{p-1-j'_1}\cup\ldots \cup y_a^{p-1-j'_a}.\]
Applying $.\lambda_l$ on both sides, we obtain
\[c.\lambda_1^{j_1}\ldots \lambda_a^{j_a}\equiv (p-1-j_l+1)\cdot c.N\cup y_1^{p-1-j_1}\cup\ldots \cup y_a^{p-1-j_a}\]
The scalar $p-j_l$ is nonzero, since $j_l\in \{0,\ldots,p-1\}$ was assumed to be at least one. This completes the induction step.
\end{proof}
Recall that $K_*$ denotes Koszul complex of $\FF G$ with respect to the sequence $(\lambda_1,\ldots,\lambda_a)$ and $K^*$ denotes its dual cochain complex.
\begin{lemma}\label{lem:compinKoszulquotient} In the spectral sequence which arises by filtering the dual of the Koszul complex $K_*$ \begin{enumerate}
\item \label{lem:compinKoszulquotienti} 
the differential $d_0$ on the $E_0$-page is zero and
\item \label{lem:compinKoszulquotientii} 
the differential $d_1$ sends the class $[z_{t}^*.\lambda_1^{p-1}\ldots \lambda_i^{p-2}\ldots \lambda_{a}^{p-1}]\in E_1^{1,|t|}$ to \[d_1([z_{t}^*.\lambda_1^{p-1}\ldots \lambda_i^{p-2}\ldots \lambda_{a}^{p-1}]) = \begin{cases} [\pm z_{t\cup \{i\}}^*.N]&\text{ if }  i \notin t\\0&\text{ otherwise.}\end{cases}\] 
\end{enumerate}
\end{lemma}
\begin{proof}
\begin{enumerate}[leftmargin=*]
    \item  By definition of the differential on the Koszul complex we have $d(K_q)\subset IK_{q-1}$.
    Recall that $F^kK^q=\{f\in K^q\mid f.\lambda =0 \text{ for all } \lambda \in I^{k+1}\}$ and $d_{ K^*}(f) = \pm f\circ d$. Thus $d_{K^*}$ maps $F^kK^q$ to $F^{k-1}K^{q+1}$ and hence the differential $d_0$ on $E_0^{k,q}=F^k K^q/F^{k-1}K^q$ is zero.
    \item The differential $d_1$ is induced by the differential $d_{K^*}$. We prove the identity already on the level of cochains. Evaluating the left-hand side on a generic $\FF$-basis element $\lambda_1^{j_1}\ldots \lambda_a^{j_a}z_{s}$ for $|s|=|t|+1$ yields:
       \begin{align*}
        & d(z_{t}^*.\lambda_1^{p-1}\ldots\lambda_i^{p-2}\ldots \lambda_a^{p-1}) (\lambda_1^{j_1}\ldots \lambda_a^{j_a}z_s)\\
        =&-(-1)^{|t|} z_{t}^* (\lambda_1^{p-1+j_1}\ldots\lambda_i^{p-2+j_i}\ldots \lambda_a^{p-1+j_i} \cdot d(z_{s}))\\
        =&\begin{cases}\pm z_t^*(N\cdot z_t) & \text{ if } s=t\cup \{i\} \mbox{ and }j_1=\ldots =j_a=0\\
        0&\text{ otherwise}\end{cases}\\        
        =&\pm z_{t\cup\{i\}}^*.N(\lambda_1^{j_1}\ldots \lambda_a^{j_a}z_s).
    \end{align*}
    Thus $d(z_{t}^*.\lambda_1^{p-1}\ldots\lambda_i^{p-2}\ldots \lambda_a^{p-1})$ is $z_{t\cup\{i\}}^*.N$ up to a sign. \qedhere
\end{enumerate}
\end{proof}

 For any free $G$-space $X$ the spectral sequence is multiplicative and $E_1^{0,*}$ is isomorphic to $H^*(X/G)$ as graded rings. The following lemma shows that if the cochain complex $C^*(X)$ is $G$-homotopy equivalent to $K^*$, then this graded ring is isomorphic to the cohomology ring of the $a$-torus.

\begin{lemma}\label{lem:prodstructureonKoszulquotient} Suppose that $K^*\simeq C^*(X)$ for a free $G$-space $X$. Then 
for $i\notin t$ the elements $[z_{t\cup \{i\}}^*.N]\in E_1^{0,|t|+1}$ and $[z_t^*.N]\cup a_i\in E_1^{0,|t|+1}$ agree up to a unit in $\FF$
and thus $[z_t^*.N]\equiv \prod_{i\in t}[z_{\{i\}}^*.N]$ in $E_1^{0,|t|}$.
\end{lemma}
\begin{proof}
By Lemma~\ref{lem:compinKoszulquotient} \eqref{lem:compinKoszulquotientii} we know that $d_1$ sends $[z_t^*.\lambda_1^{p-1}\ldots \lambda_i^{p-2}\ldots \lambda_a^{p-1}]\in E_1^{1,|t|}$ to $\pm [z_{t\cup\{i\}}^*.N]\in E_1^{0,|t|+1}$. We have 
\[[z_t^*.\lambda_1^{p-1}\ldots \lambda_i^{p-2}\ldots \lambda_a^{p-1}]\equiv [z_t^*.N]\cup y_i\]
by Proposition~\ref{prop:findingyi}. It follows from Theorem~\ref{mthm:1} \eqref{mthm:1iii} that $d_1([z_t^*.N]\cup y_i)=\pm [z_t^*.N]\cup a_i$. This proves the first statement and we proceed to the second statement.
    
    Recursively, it follows that 
    $[z_t^*.N]\equiv [z_\emptyset^*.N]\prod_{i\in t} a_i$.
    
    Since $0\neq [z_\emptyset^*.N]\in E^{0,0}_1\cong \FF$, it differs from the neutral element of the $\cup$-product just by a unit. Thus $a_i\equiv [z_\emptyset^*.N]\cup a_i\equiv [z_{\{i\}}^*.N]$ by the first statement. The second statement follows. \qedhere
\end{proof}
 The complexes with small homology arise as mapping cones of chain maps given by multiplication with specific cycles of degree two in the Koszul complex. We consider more generally mapping cones arising from multiplication with arbitrary cycles of any degree $r\geq 0$. If we change the cycle by adding a boundary, the resulting chain map only changes up to homotopy and thus the mapping cones are homotopy equivalent as $\FF G$-modules.
Thus without loss of generality, we can assume that the $r$-dimensional cycle is of the form
\begin{equation}\label{eq:defofw}
w=\sum_{s} \mu_s (\lambda_{s_1}^{p-1}\ldots \lambda_{s_r}^{p-1}) z_{s}\in K_r,\end{equation}
where $s$ runs through the set of increasing $r$-tuples of elements of $\{1,\ldots,a\}$, as previously $s_i$ denotes the $i$-th entry of such a tuple, and $\mu_s\in \FF$ is some coefficient. 

Left multiplication with the cycle $w$ induces a map
\[w\colon \Sigma^r K_*\to K_*
\]
and we denote its mapping cone by $\Cone(w)_*$.
We need to work cohomologically to make use of the multiplicative structure in the spectral sequence from Theorem~\ref{thm:main1}. Therefore, we consider the dual cochain complexes $K^*$ and $\Cone(w)^*$. Explicitly $\Cone(w)^m= K^m \oplus K^{m-r-1}$ and the differential $\Cone(w)^m\rightarrow \Cone(w)^{m+1}$ is
\[\begin{pmatrix} d_{K^*} & 0 \\
(-1)^m w^* & d_{K^*}
\end{pmatrix},
\]
where $d_{K^*}:K^m\rightarrow K^{m+1}$ is given by $f\mapsto -(-1)^{|f|} f\circ d_{K_*}$ and $w^*$ is the map given by precomposition with multiplication by $w$. 

 Explicitly, the map $w^*$ sends a vector  $z_t^*=(z_{t_1}\wedge\ldots \wedge z_{t_m})^*$ of the basis for $K^m$ to
\begin{equation} \label{eq:wstar}z_{t}^*(w\wedge \_)
=\sum_{s}  z_{t}^*(z_{s}\wedge \_).\left(\mu_s\lambda_{s_1}^{p-1}\ldots \lambda_{s_r}^{p-1}\right).\end{equation}

If $s\not\subseteq t$, then the map $z_t^*(z_s\wedge \_)$ in that summand is zero. Otherwise the map is $\pm (z_{t\setminus s})^*$.

The next observation follows immediately from the definition of the differentials on the pages of the spectral sequence of a filtered complex. We will use it to compute the first pages of the spectral sequence of $\Cone(w)^*$.

\begin{remark}\label{rem:changingthediffs} Let $(C^*,d)$ be a filtered cochain complex. Suppose we have a second way of equipping these graded modules with differentials $d'$ and assume that $d-d'$ decreases the filtration degree by $m$. Then the first $m-1$ pages (including the differentials) of the spectral sequence associated to the filtered cochain complexes $(C^*,d)$ and $(C^*,d')$ agree. The entries on the $E_m$-page are still the same, but the differential might be different.

In our situation the spectral sequence arises by filtering free $\FF G$-cochain complexes $C^*$ by powers of the augmentation ideal $I$. Thus if we change the differential by adding a map of degree one factoring through $C^*.I^m$, the pages $E_1,\ldots,E_{m-1}$ do not change.

In particular, the spectral sequences for $\Cone(w)^*$ and for $\Cone(0)^*$ of the zero map $0\colon \Sigma^r K_*\to K_*$ agree up to page $E_{r(p-1)-1}$ including the differentials.
\end{remark}
The following lemma computes parts of the potential product structure on the $E_1$-page.

\begin{lemma}\label{lem:prodstructureoncone} Let $w$ be of degree $r\ge 2$. Suppose that $\Cone(w)^*\simeq C^*(X)$ for a free $G$-space $X$. Then 
for $i\notin t$ the elements $[(z_{t\cup \{i\}}^*.N,0)]\in E_1^{0,|t|+1}$ and $[(z_t^*.N,0)]\cup a_i\in E_1^{0,|t|+1}$ agree up to a unit in $\FF$
and thus $[(z_t^*.N,0)]\equiv \prod_{i\in t}[(z_{\{i\}}^*.N,0)]$ in $E_1^{0,|t|}$.
\end{lemma}
\begin{proof}
Since $r\ge 2$, it follows from Remark~\ref{rem:changingthediffs} that the $E_1$-page is a direct sum of the spectral sequence for $K^*$ and for $\Sigma^{r+1} K^*$. As the action of the $a_i$'s can be calculated from the differential on the $E_1$-page, we obtain as in the proof of Lemma~\ref{lem:prodstructureonKoszulquotient} that
\[[(z_t^*.N,0)]\cup a_i=[(z_{t\cup\{i\}}^*.N,0)].\]
The element $[(z_\emptyset^*.N,0)]\in E_1^{0,0}\cong \FF$ differs from the neutral element of the bigraded ring $E_1^{*,*}$ only by a unit and thus
\[[(z_t^*.N,0)] \equiv \prod_{i\in t}[(z_{\{i\}}^*.N,0)]\]
as in the proof of Lemma~\ref{lem:prodstructureonKoszulquotient}.
\end{proof}

The following lemma exhibits a nonzero differential in the spectral sequence from which we will deduce a contradiction to the Leibniz rule.

\begin{lemma}\label{lem:compwithw} Consider the spectral sequence obtained by the filtration of $\Cone(w)^*$. Suppose that $s$ is an $r$-tuple whose coefficient $\mu_s\in \FF$ in the definition \eqref{eq:defofw} of $w$ is nonzero. Then
\begin{enumerate}
\item \[d(z_{s}^*.\prod_{i\notin s}\lambda_i^{p-1},0)=\pm \mu_s(0,z_\emptyset^*.N),\]
\item
the cochain $(z_{s}^*.\prod_{i\notin s} \lambda_i^{p-1},0)$ represents a class in $E_{r(p-1)}^{r(p-1),r}$, the differential $d_{r(p-1)}$ sends it to $\pm \mu_s[(0,z_\emptyset^*.N)]\in E_{r(p-1)}^{0,r+1}$, and
\item the class $[(0,z_\emptyset^*.N)]\in E_{r(p-1)}^{0,r+1}$ is nonzero.
\end{enumerate}
\end{lemma}
\begin{proof}
\begin{enumerate}[leftmargin=*]
    \item Direct calculation shows that the differential $d_{K^*}$ sends the left entry to zero and precomposing with $\wedge$-multiplication with $w$ yields $\pm \mu_s z_\emptyset^*.\lambda_1^{p-1}\ldots \lambda_a^{p-1}$ by \eqref{eq:wstar}.
    
    \item The cochain $(z_{s}^*.\prod_{i\notin s} \lambda_i^{p-1},0)\in \Cone(w)^{r}$ lives in filtration degree $r(p-1)$ since $\prod_{i\notin s} \lambda_i^{p-1}\in I^{(a-r)(p-1)}$, and its boundary lives in filtration degree $0$.
    
    Thus that boundary represents a class in $E^{0,r+1}_{r(p-1)}$ and $(z_{s}^*.\prod_{i\notin s} \lambda_i^{p-1},0)$ represents a class in $E^{r(p-1),r}_{r(p-1)}$.    
    
    \item We compare the spectral sequence for $\Cone(w)^*$ with the spectral sequence for $\Cone(0)^*$, where $0$ denotes the zero map $\Sigma^{r}K_*\rightarrow K_*$. In the latter the cochain $(0,z_\emptyset^*.\lambda_1^{p-1}\ldots \lambda_a^{p-1} )$ survives to the $E_\infty$-page, since that spectral sequence is a direct sum of the spectral sequence for $K^*$ and for $\Sigma^{r+1}K^*$.
    By Remark~\ref{rem:changingthediffs}, this cochain survives to the $E_{r(p-1)}$-page in the spectral sequence of $\Cone(w)^*$ as well. \qedhere
\end{enumerate}
\end{proof}
We deduce the main result of this section from Lemma~\ref{lem:prodstructureoncone} and Lemma~\ref{lem:compwithw}.
\begin{theorem}\label{thm:mainnotreal} Let $w\in K_r$ represent a nonzero homology class with $r\ge 2$. Then there is no free $G$-space $X$ such that $\Cone(w)^*$ and $C^*(X)$ are homotopy equivalent as cochain complexes over $\FF G$. 
\end{theorem}
\begin{proof}
Suppose there exists a free $G$-space $X$ such that $\Cone(w)^*$ and $C^*(X)$ are equivariantly homotopy equivalent. Write $w$ as in \eqref{eq:defofw}. Since $w$ represents a nonzero homology class, we can choose an index $s$ such that $\mu_s$ is nonzero. We will show that the class $[(z_{s}^*.\prod_{i\notin s} \lambda_i^{p-1},0)]$ from Lemma~\ref{lem:compwithw} can be written as a product of permanent cycles. Thus it has to be a permanent cycle as well by the multiplicativity of the spectral sequence. This contradicts Lemma~\ref{lem:compwithw}.

By Lemma~\ref{lem:prodstructureoncone}, we have $[(z_s^*.N,0)]\equiv \prod_{i\in s} [(z_{\{i\}}^*.N,0)]$ in $E_1^{0,r}$. 

It follows from Proposition~\ref{prop:findingyi} that 
\begin{align*}[(z_s^*,0).\prod_{i\notin s}\lambda_i^{p-1}] &\equiv [(z_s^*,0).N]\cup \prod_{i\in s}y_i^{p-1}\in E^{r(p-1),r}_{1},\\
[(z_{\{i\}}^*,0).\prod_{j\neq i}\lambda_j^{p-1}] &\equiv [(z_{\{i\}}^*,0).N]\cup y_i^{p-1}\in E^{p-1,1}_{1}. \end{align*}
Thus $[(z_s^*.\prod_{i\notin s}\lambda_i^{p-1},0)]\in E^{r(p-1),r}_{1}$ can be written as $\prod_{i\in s}[(z_{\{i\}}^*,0).\prod_{j\neq i}\lambda_j^{p-1}]$ up to a unit.

Using Remark~\ref{rem:changingthediffs}, we compare again the spectral sequence with the spectral sequence of $\Cone(0)^*$. In the latter, each factor $[(z_{\{i\}}^*,0).\prod_{j\neq i}\lambda_j^{p-1}]$ survives to the $E_\infty$-page and thus it survives here to the $E_{r(p-1)}$-page. In fact, its image under the  differential $d_q$ for $q\ge r(p-1)$ lies in $E_q^ {p-1-q,2}$ which is the zero group since $p-1-q\le (p-1)-r(p-1)<0$ using that $r\ge 2$. Thus it is a permanent cycle in the spectral sequence for $\Cone(w)^*$ as well.
The differential $d_{r(p-1)}$ sends the product $[(z_s^*.\prod_{i\notin s}\lambda_i^{p-1},0)]$ of these factors to a nonzero element by Lemma~\ref{lem:compwithw}, contradicting the Leibniz rule. Thus $\Cone(w)_*$ can not be realized topologically.
\end{proof}

The following corollary provides the nonrealizability part of Theorem~\ref{mthm}.

\begin{corollary}\label{cor:mainthm}Let $w\in K_r$ represent a nonzero homology class with $r\ge 2$. There is no topological space with a free $G$-action whose singular chain complex with $\FF$-coefficients is quasi-isomorphic to $\Cone(w)_*$ as $\FF G$-chain complexes.
\end{corollary}
\begin{proof}
Let $X$ denote the simplicial set of singular simplices of a topological space with a free $G$-action. Then $X$ inherits a free action and we can replace the singular chain complex of the topological space by the normalized chain complex $C_*(X)$. If $C_*(X)$ and $\Cone(w)_*$ are connected by a zig-zag of quasi-isomorphisms, then they are chain homotopy equivalent $\FF G$-chain complexes since they are bounded below and free; see \cite[Theorem~10.4.8]{weibel1994}. Dualizing, we obtain an equivariant chain homotopy equivalence $C^*(X)\simeq \Cone(w)^*$. This contradicts Theorem~\ref{thm:mainnotreal}.
\end{proof}

Carlsson's approach to nonexistence of equivariant Moore spaces using an interaction between group cohomology and Steenrod operations can be adapted to realizability questions of equivariant chain complexes, but does not provide obstructions for the realizability of $\Cone(w)_*$.

\begin{remark}\label{rem:steenrodapproach} Carlsson showed that there exist $G$-modules $M$ with no equivariant Moore space \cite{carlsson1981counterexample} based on the following strategy. He constructed an annihilator ideal in $H^*(BG;\FF_p)$ that can be calculated purely algebraically and which must be invariant under the action of the Steenrod algebra if $M$ admits an equivariant Moore space. This idea does not provide obstructions for the realizability of $\Cone(w)_*$ in general. When $p=2$, then $H^*(BG,\FF_2)$ is a polynomial algebra in $a$ generators $a_i$ of degree $1$ and we can compute the $H^*(BG,\FF_2)$-action on $H^*(X/G;\FF_2)\cong H^*(\Cone(w)^*/ \Cone(w)^*.I)$ for a potential free $G$-space $X$ realizing $\Cone(w)_*$ purely algebraically from the $d_1$-differential of the spectral sequence. If $X$ realizes $\Cone(w)_*$, then the annihilator ideal 
\[\{b\in H^*(BG;\FF_2)\mid x\cup b=0 \text{ for all } x\in H^*(X/G)
\}\]
must be invariant under the Steenrod action. For the chain complex of Example~\ref{ex:smallnonrealizable}, this annihilator ideal is $(a_1^2,a_2^2)$. Thus it is invariant under the Steenrod operations and does not provide obstructions for realizability.
\end{remark}
For fixed $p$ and $a$ the best known lower bound for the total dimension of $H_*(X;\FF_p)$ for finite, nonempty, free $(\ZZ/p)^a$-CW complexes $X$ agrees with the best known lower bound for the total dimension of $H_*(C_*)$ for finite, nonacyclic, free $\FF_p(\ZZ/p)^a$-chain complexes $C$. The nonrealizability of the counterexamples of Iyengar and Walker suggests that the two minima of the total dimensions may be different. If a cochain complex arises as the cochains of a finite, free $(\ZZ/p)^a$-CW complex, we obtain additional structure on the spectral sequence from Theorem~\ref{mthm:1}. For instance the Leibniz rule provides restrictions on the differentials appearing on each page. One may hope that this can be used to show that more classes survive.

In general the surjection operad $\mathcal{S}$ acts on the cochain complex of a simplicial set. In this paper we use only very few of those operations, namely $\cup$ and $\cup_1$.
By \cite[Section~3]{mcclure2003multivariable}, the surjection operad is filtered by suboperads $\mathcal{S}_n$ which are quasi-isomorphic in the category of chain operads to the singular chains on the little $n$-cubes operad.
 The operation $\cup$ lies in $\mathcal{S}_1$ and $\cup_1$ lies in $\mathcal{S}_2$; this leads to the following question.
 
 \begin{question}
  What is the smallest total dimension of $H^*(C^*)$ for finite, free, nonacyclic $\FF (\IZ/p)^a$-cochain complexes $C^*$ equipped with a $(\IZ/p)^a$-equivariant action of $\mathcal{S}_n$ for a given integer $n\ge 1$?
\end{question}

\begin{remark} In light of Remark~\ref{rem:hopfmodulestructure} we could also require a compatible $C^0(G)$-Hopf module structure on the cochain complex $C^*$ appearing in the last question. Compatibility means that the left coaction of $C^0(G)$ is induced by the $G$-action via
\[C^*\to C^0(G)\otimes C^*,\qquad x\mapsto \sum_{g\in G} g^*\otimes x.g.\] 
\end{remark}

 It would be interesting to use the $\mathcal{S}$-action to equip the spectral sequence with additional structure.  For instance one could ask, whether the Steenrod operations on $H^*(X;\FF)$ for a free $G$-space $X$ induce operations $E_r^{k,q}\rightarrow E_{r'}^{k',q'}$ for suitable $r',k',q'$.

\section{Realizable mapping cones} \label{sec:rmappingcones} As in Section~\ref{sec:nrmappingcones}, let $G=(\IZ/p)^a$, we identify $\FF G\cong \FF[\lambda_1,\ldots,\lambda_a]/(\lambda_1^p,\ldots,\lambda_a^p)$, and denote the Koszul complex with respect to the sequence $(\lambda_1,\ldots,\lambda_a)$ by $K_*$. Let $w\in K_*$ be a cycle. In Section~\ref{sec:nrmappingcones}, we established nonrealizability of $\Cone(w)_*$ if $w$ represents a nonzero homology class of degree at least $2$. In this section, we consider the realizability of $\Cone(w)_*$ as a chain complex $C_*(X;\FF)$ of a free $G$-space $X$. Any chain complex $C_*(X;\FF)$ must be obtained from $C_*(X;\FF_p)$ via induction. Working over $\FF_p$, we will show that there are no additional obstructions for realizing $\Cone(w)_*$.

\begin{theorem}\label{thm:realize} Let $w\in K_r$ be a cycle of degree $r$. If $r\leq 1$ or if $w$ is a boundary of any degree, then there exists a finite, free $G$-CW complex $X$ such that its cellular chain complex with $\FF_p$-coefficients $C^{\mbox{cell}}_*(X;\FF_p)$ is homotopy equivalent to the mapping cone $\Cone(w:\Sigma^rK_*\rightarrow K_*)_*$ as $\FF_pG$-chain complexes.
\end{theorem}
We will prove Theorem~\ref{thm:realize} at the end of this section based on the following special cases which still hold over an arbitrary field $\FF$ of characteristic $p$. The group $G=(\IZ/p)^a$ acts on the $a$-torus $T^a$ by the inclusion $(\IZ/p)^a\subset T^a$ as a subgroup. We will use repeatedly the observation that $K_*\cong C^{\mbox{cell}}_*(T^a;\FF)$.
\begin{example} \label{ex:w0}
Suppose that $w$ represents the zero homology class of degree $r$. Then the product $X=S^{r+1}\times T^a$ of the sphere $S^{r+1}$ with trivial $(\IZ/p)^a$-action and the $a$-torus $T^a$ with free action as above realizes $\Cone(w)_*$.
\end{example}
\begin{example}\label{ex:wdegree0}
If $r=0$ and $w\in K_0$ represents a nonzero homology class, then $[w]\in H_0(K_*)\in \FF$ is a unit. Thus multiplication with $w$ induces an isomorphism on homology. Hence $\Cone(w)_*$ is a contractible chain complex and realized by the empty space $X=\emptyset$.
\end{example}
The following example is based on the construction of Lens spaces.
\begin{example}\label{ex:simplewindeg1ais1} 
Let $r=1$, $w=\lambda_1^{p-1}z_1$ and consider an elementary abelian $p$-group of rank $a=1$. The group $S^1$ of complex units of absolute value $1$ acts by multiplication in each coordinate on $S^{3}\subset \mathbb{C}^2$ and we restrict this action to $\IZ/p\subset S^1$. We show that $S^3$ realizes $\Cone(w)_*$. There is an equivariant CW-structure on $S^3$ with only one equivariant cell in each dimension $0,1,2,3$ and with equivariant chain complex (cf.~\cite[\S 28]{cohen1973})
    \[0\rightarrow \FF\IZ/p\stackrel{\lambda_1}{\rightarrow} \FF\IZ/p\stackrel{\lambda_1^{p-1}}{\longrightarrow}
    \FF\IZ/p\stackrel{\lambda_1}{\rightarrow}
    \FF\IZ/p\rightarrow 0.\]
    The only difference between this chain complex and the mapping cone $\Cone(w)_*$ is the sign of the differentials and thus both are isomorphic chain complexes.
\end{example}
We extend Example~\ref{ex:simplewindeg1ais1} to elementary abelian $p$-groups of arbitrary rank $a$.
\begin{example}\label{ex:wdeg1special} 
Let $r=1$, $w=\lambda_1^{p-1}z_1$ and $a$ be arbitrary. We will show that the $G$-CW complex $X=S^3\times T^{a-1}$ realizes $\Cone(w)_*$, where the action of the group $G=\IZ/p \times (\IZ/p)^{a-1}$ is the product action of Example~\ref{ex:simplewindeg1ais1} and the action given by the inclusion $(\IZ/p)^{a-1}\subset T^{a-1}$.

The decomposition $(\ZZ/p)^a=\ZZ/p\times (\ZZ/p)^{a-1}$ yields an isomorphism between the group algebra $\FF (\ZZ/p)^a$ and the tensor product of the algebras $\FF \ZZ/p\cong \FF[\lambda_1]/(\lambda_1^p)$ and $\FF (\ZZ/p)^{a-1}\cong\FF[\lambda_2,\ldots,\lambda_a]/(\lambda_2^p,\ldots,\lambda_a^p)$.

We write $K^{\{1\}}_*$ for the Koszul complex of $\lambda_1$ in the ring $\FF \IZ/p$ and
$K^{\{2,\ldots,a\}}_*$ for the Koszul complex of the sequence $\lambda_2,\ldots,\lambda_a$ in the ring $\FF (\ZZ/p)^{a-1}$. The isomorphism $\FF \ZZ/p\otimes_\FF \FF (\ZZ/p)^{a-1}\cong \FF(\ZZ/p)^a$ induces an isomorphism of dgas $K_*^{\{1\}}\otimes_\FF K_*^{\{2,\ldots,a\}}\rightarrow K_*$. 
Using this isomorphism, we conclude that the equivariant cellular chain complex of $S^3\times T^{a-1}$ realizes $\Cone(w)_*$ since
\begin{align*}
C^{\mbox{cell}}_*(S^3\times T^{a-1})&\cong C^{\mbox{cell}}_*(S^3)\otimes_\FF C^{\mbox{cell}}_*(T^{a-1})\\
&\cong\Cone(\lambda_1^{p-1}z_1\colon \Sigma K_*^{\{1\}}\rightarrow K_*^{\{1\}})_* \otimes_\FF K^{\{2,\ldots,a\}}_*\\
&= \Cone((\lambda_1^{p-1}z_1\colon \Sigma K_*^{\{1\}}\rightarrow K_*^{\{1\}})_*\otimes_\FF \id_{K_*^{\{2,\ldots,a\}}})\\
&= \Cone(\lambda_1^{p-1}z_1\otimes_\FF 1\colon \Sigma K_*^{\{1\}}\otimes_\FF K_*^{\{2,\ldots,a\}}\rightarrow K_*^{\{1\}}\otimes_\FF K_*^{\{2,\ldots,a\}})_*\\
&\cong \Cone(\lambda_1^{p-1}z_1\colon \Sigma K_*\rightarrow K_*)_*.
\end{align*}
\end{example}

Working over $\FF_p$, we will realize $\Cone(w)_*$ for an arbitrary cycle $w\in K_1$ with $[w]\neq 0 \in H_1(K_*)$ by pulling back the group action on $S^3\times T^{a-1}$ along an automorphism of $G$. The following two lemmas help us find the desired automorphism of $G$. 
For $\varphi\in \aut(G)$, we denote the induced ring automorphism of $\FF_p G$ by $\varphi$ as well.
\begin{lemma}\label{lem:autGtrans} Let $G=(\ZZ/p)^a$ and $I\subset \FF_p G$ be the augmentation ideal of the group ring over $\FF_p$. Then $\aut(G)$ acts transitively on the set of all nonzero elements in $I/I^2$.
\end{lemma}
\begin{proof} 
Consider an arbitrary nonzero class $\mu_1\lambda_1+\ldots+\mu_s\lambda_s+I^2$. We construct an automorphism $\varphi$ of $G$, such that $\varphi(\lambda_1+I^2)=\mu_1\lambda_1+\ldots+\mu_s\lambda_s+I^2$. Since $\aut(G)$ acts transitively on the set of all nonzero elements in $G=(\IZ/p)^a$, we can choose $\varphi$ such that $\varphi(f_1)=f_1^{\mu_1}\cdot \ldots \cdot f_a^{\mu_a}$, where $f_1,\ldots,f_a$ are the canonical generators of $(\IZ/p)^a$. Then 
\begin{align*}
    \varphi(\lambda_1+I^2)&=\varphi(f_1-1)+I^2=\varphi(f_1)-1+I^2\\
    &=f_1^{\mu_1}\cdot \ldots \cdot f_a^{\mu_a} -1 +I^2\\
    &=(\lambda_1+1)^{\mu_1}\cdot \ldots\cdot(\lambda_a+1)^{\mu_a}-1+I^2\\
    &=\mu_1\lambda_1+\ldots+\mu_a\lambda_a+I^2\qedhere.
\end{align*}
\end{proof}
The following lemma calculates the homology class of a certain cycle.
\begin{lemma} \label{lem:homologyclassofcycle} Let $\lambda$ be an element of the augmentation ideal $I\subset \FF G$. Then there is a chain $z$ in $K_1$ with $dz =\lambda$ in $K_0=\FF G$. Moreover, the homology class of the cycle $\lambda^{p-1}z$ is independent of the choice of $z$ and it is given by
\[\mu_1^p[\lambda_1^{p-1}z_1]+\ldots+\mu_a^p[\lambda_a^{p-1}z_a],\]
where $\mu_i\in \FF$ is determined by $\lambda+I^2 = \mu_1\lambda_1+\ldots +\mu_a\lambda_a+I^2$. 
\end{lemma}
\begin{proof}
Before constructing $z$, we show that the homology class of $\lambda^{p-1}z$ does not depend on the choice of $z$. The element $\lambda^{p-1}z$ is a cycle, because 
\[d(\lambda^{p-1}z) =\lambda^p=0.\] For any other choice $z'$ with $dz'=\lambda$, the difference $z-z'$ is a cycle. Any element of $I$ acts trivially on $H_*(K)$; see \cite[Proposition~1.6.5 (b)]{brunsherzog1993}. It follows that $[\lambda^{p-1}(z-z')]=0$ and thus 
\[[\lambda^{p-1}z'] =[\lambda^{p-1}z]. \]
A choice of $z$ can be constructed in the following way. Write $\lambda$ as an $\FF$-linear combination of monomials $\lambda_1^{j_1}\ldots\lambda_a^{j_a}$
with exponents less than $p$ and not all equal to zero. For each such summand, choose an index $k$ such that $j_k>0$ and replace the summand by $\lambda_1^{j_1}\ldots\lambda_k^{j_k-1}\ldots\lambda_a^{j_a} z_{k}$. Then the resulting chain $z$ in $K_1$ satisfies $dz=\lambda$ by construction. It remains to compute the homology class of $\lambda^{p-1}z$. Let $V\subset K_1$ be the submodule spanned by all elements $\lambda_iz_j$ for $i\neq j$.
Recall that $H_1(K_*)$ is an $a$-dimensional $\FF$-vector space generated by $[\lambda_i^{p-1}z_i]$ for $1\leq i\leq a$. The composite 
\[\FF^a\cong H_1(K_*)= \frac{\ker(d\colon K_1\to K_0)}{\langle \lambda_iz_j -\lambda_j z_i\mid i\neq j\rangle} \hookrightarrow \frac{K_1}{\langle \lambda_iz_j -\lambda_j z_i\mid i\neq j\rangle}\twoheadrightarrow  \frac{K_1}{V}\]
sends $(\gamma_1,\ldots,\gamma_a)$ to $\sum_i \gamma_i \lambda_i^{p-1} z_i+V$ and it is straightforward to see that it is injective. Thus it suffices to show that the coset of $\lambda^{p-1}z$ in $K_1/V$ is \[\mu_1^p\lambda_1^{p-1}z_1+\ldots+\mu_1^p\lambda_1^{p-1}z_1+V,\]
where $\lambda =\mu_1\lambda_1+\ldots +\mu_a\lambda_a+r$ with $r\in I^2$. Let $J\subset \FF G$ be the ideal spanned by all elements of the form $\lambda_i\lambda_j$ for $i\neq j$. Since $I^p$ is spanned by all monomials of degree at least $p$ in which every exponent is less than $p$, it follows that $I^p\subset J$. We deduce that
\[\lambda^{p-1} \in  \mu_1^{p-1}\lambda_1^{p-1}+\ldots +\mu_a^{p-1}\lambda_a^{p-1}+J\]
since $\lambda_i^k r^{p-1-k}\in I^{k+2(p-1-k)}\subset J$ for $k<p-1$ and since $\lambda_i\lambda_j\in J$ by definition. By construction of $z$, we obtain
\[z\in \mu_1z_1+\ldots \mu_az_a + I\cdot K_1.\]
It follows that
\[\lambda^{p-1}z= (\mu_1^{p-1}\lambda_1^{p-1}+\ldots \mu_a^{p-1}\lambda_a^{p-1}+\lambda')(\mu_1z_1+\ldots \mu_az_a+z')\mbox{ with } \lambda'\in J, z'\in I\cdot K_1.\] 
Expanding the product and using that
\begin{align*}
    \lambda'\cdot \mu_k z_k &\in J\cdot K_1\subset V,\\
    \lambda'\cdot z' &\in J\cdot K_1\subset V,\\
    \lambda_i^{p-1}\cdot\mu_j z_j&\in V\mbox{ for }i \neq j,\\
    \lambda_i^{p-1}z'&\in I^{p-1}\cdot I\cdot K_1\subset J\cdot K_1\subset V,
\end{align*}
we see that the coset of $\lambda^{p-1}z$ in $K_1/V$ is given by 
\[\mu_1^p\lambda_1^{p-1}z_1+\ldots+\mu_a^p\lambda_a^{p-1}z_a+V\]
as desired.
\end{proof}

We use the canonical isomorphism $H_1(K_*)\cong I/I^2, \quad [\lambda_i^{p-1}z_i]\mapsto \lambda_i$, in the following result.
\begin{proposition}\label{prop:actiononkoszulhomology}
Restriction of scalars of $K_*$ along automorphisms of $\FF_p G$ induces a right $\aut(\FF_p G)$-action $\rho$ on the graded $\FF_p$-algebra $H_*(K_*)$ such that
\begin{equation}\label{eq:actiononH1}
\xymatrix{ H_1(K_*)\ar[r]^{\rho(\varphi)}\ar[d]_{\cong} & H_1(K_*)\ar[d]^{\cong} \\
I/I^2 \ar[r]^{\varphi^{-1}} & I/I^2}
\end{equation}
commutes for all $\varphi\in \aut(\FF_p G)$.
\end{proposition}
\begin{proof}
Let $\varphi\in \aut(\FF_p G)$ be an automorphism of the algebra $\FF_pG$. We will construct an automorphism of the dga $K_*$ over $\FF_p$ and denote the induced automorphism of $H_*(K_*)$ by $\rho(\varphi)$.

For any $\FF_p G$-module $M$, restriction of scalars along $\varphi$ yields an $\FF_p G$-module $\varphi^* M$ with the same underlying vector space, but new action $\beta \otimes m \mapsto \varphi(\beta) m$ which we denote by $\beta\ast m$. Moreover, applying $\varphi^*$ degreewise to $K_*$ yields a dga $\varphi^*K_*$ with unit $\FF_p G\to \varphi^*K_0, \quad \beta\mapsto \beta\ast 1 = \varphi(\beta)$ and the original multiplication. Explicitly, the differential $\varphi^* K_n \to \varphi^* K_{n-1}$ on $\varphi^*(K_*)$ is given by 
\[z_s\mapsto \sum_{i=1}^n (-1)^i \varphi^{-1}(\lambda_{s_i}) \ast z_{s\setminus \{s_i\}}.\]
We obtain an isomorphism of dgas over $\FF_p G$ from $\varphi^*K_*$ to the Koszul complex $K_*(\varphi^{-1}(\lambda_1),\ldots,\varphi^{-1}(\lambda_a))$ of the sequence $\varphi^{-1}(\lambda_1),\ldots,\varphi^{-1}(\lambda_a)$ sending 
\[\beta \ast z_s\mapsto \beta \cdot z_{s}.\]
In particular, in degree $0$ the isomorphism $\varphi^*\FF_p G\rightarrow \FF_p G$ maps $\beta=\varphi^{-1}(\beta)\ast 1$ to $\varphi^{-1}(\beta)$.
By \cite[Chapter~I.1.6, p.~52]{brunsherzog1993}, expressing $\varphi^{-1}(\lambda_i)$ as a linear combination
\[\sum_{j=1}^a a_{ji} \lambda_j
\]
for a choice of $a_{ji}\in \FF_p G$ induces an isomorphism \[K_1(\varphi^{-1}(\lambda_1),\ldots, \varphi^{-1}(\lambda_a))\to K_1(\lambda_1,\ldots, \lambda_a),\quad z_i\mapsto \sum_{j} a_{ji}z_j,\] 
that commutes with the differentials to $\FF_p G$, and hence extends to an isomorphism of the entire Koszul complexes.
We denote the composite of the two dga isomorphisms by $\Phi\colon \varphi^*(K_*)\rightarrow K_*$. The underlying dga of $\varphi^*(K_*)$ over $\FF_p$ is just $K_*$ itself. Let $\rho(\varphi)\colon H_*(K_*)\to H_*(K_*)$ be the graded algebra automorphism induced by $\Phi$. We show that the square \eqref{eq:actiononH1} commutes. Consider one of the generators $[\lambda_i^{p-1}z_i]$ of $H_1(K_*)$. By definition of $\Phi$, we have
\[\Phi(\lambda_i^{p-1}z_i)= \Phi(\varphi^{-1}(\lambda_i^{p-1})\ast z_i)= \varphi^{-1}(\lambda_i^{p-1}) \Phi(z_i) = \varphi^{-1}(\lambda_i)^{p-1}\Phi(z_i)\]
and
\[d(\Phi(z_i))=\Phi(d(z_i))= \Phi(\lambda_i)=\Phi(\varphi^{-1}(\lambda_i)\ast 1)=\varphi^{-1}(\lambda_i).\]
Hence $\Phi(z_i)$ is a chain with boundary $\varphi^{-1}(\lambda_i)$ so that by Lemma~\ref{lem:homologyclassofcycle}, the homology class represented by $\Phi(\lambda_i^{p-1}z_i)=\varphi^{-1}(\lambda_i)^{p-1}\Phi(z_i)$ is
\[\mu_1^p[\lambda_1^{p-1}z_1]+\ldots+\mu_a^p[\lambda_a^{p-1}z_a],\]
where $\mu_j\in\FF_p$ is determined by $\varphi^{-1}(\lambda_i)+I^2 = \mu_1\lambda_1+\ldots +\mu_a\lambda_a+I^2$. Since $\mu_j^p=\mu_j$ in $\FF_p$, it follows that the square \eqref{eq:actiononH1} commutes.

Since $H_*(K_*)$ is the exterior algebra on $H_1(K_*)$, it follows that $\rho(\varphi)$ is independent of the choices in the construction of $\Phi$. Moreover, the commutativity of \eqref{eq:actiononH1} implies that 
$\rho(\varphi_1)\circ \rho(\varphi_2) = \rho(\varphi_2\circ \varphi_1)$ on $H_1(K_*)$, and hence on $H_*(K_*)$, for two automorphisms $\varphi_1,\varphi_2$ of $\FF_pG$. Thus this construction defines indeed a right action of $\aut(\FF_pG)$ on $H_*(K_*)$.
\end{proof}
 We are ready to prove the main theorem of this section.
\begin{proof}[Proof of Theorem~\ref{thm:realize}] If $w$ is a boundary or represents a nonzero homology class in degree $0$, then $\Cone(w)_*$ can be realized as the cellular chain complex of a finite, free $G$-CW complex by Examples~\ref{ex:w0} and \ref{ex:wdegree0}.  
It remains to consider the case of $r=1$ where $[w]\in H_1(K_*)$ is an arbitrary nonzero homology class. We can choose a representative of the form $w=\sum_s \mu_s \lambda_s^{p-1}z_s$ with $\mu_s\in \FF_p$. Since $[w]$ is nonzero, not all $\mu_s$ are zero.

By Lemma~\ref{lem:autGtrans}, there is an automorphism $\varphi$ of $G$ such that \[\varphi^{-1}(\lambda_1)+I^2 = \mu_1\lambda_1+\ldots+\mu_a\lambda_a+I^2.\label{eq:defofvarphi}\]

We will realize $\Cone(w)_*$ with $\varphi^*(S^3\times T^{a-1})$, the CW complex from Example~\ref{ex:wdeg1special}, but with group action obtained by pulling back the action along $\varphi:G\rightarrow G$.

Since taking cellular chain complexes commutes with $\varphi^*$, the cellular chain complex of $\varphi^*(S^3\times T^{a-1})$ is obtained from Example~\ref{ex:wdeg1special} by restriction of scalars along $\varphi$, i.e.  $C_*^{\mbox{cell}}(\varphi^*(S^3\times T^{a-1}))=\varphi^*(\Cone(\lambda_1^{p-1} z_1)_*)$. Taking mapping cones and taking suspensions commute with restriction of scalars so that
\[\varphi^*(\Cone(\lambda_1^{p-1} z_1)_*)=\Cone(\lambda_1^{p-1}z_1\colon \Sigma \varphi^*K_*\rightarrow \varphi^*K_*)_*.\]
Using the isomorphism $\Phi$ from the proof of Proposition~\ref{prop:actiononkoszulhomology}, we deduce that this mapping cone is isomorphic to the mapping cone of $ \Sigma K_*\rightarrow K_*$ given by multiplication with $\Phi(\lambda_1^{p-1}z_1)$ as $\FF_p G$-chain complexes.

Thus it remains to prove that the homology class of $\Phi(\lambda_1^{p-1}z_1)$ is $[w]$. This holds by Proposition~\ref{prop:actiononkoszulhomology} and the choice of $\varphi$.
\end{proof}

Working over a general field $\FF$ of characteristic $p$, would require to investigate which chain complex over $\FF G$ is chain homotopy equivalent to a chain complex obtained by induction from an $\FF_pG$-chain complex. This descent question would lead in a different and purely algebraic direction.

A more general topological realization problem is to consider semifree dg-modules over the dga $K_*$.
\begin{remark}\label{rem:semifree}
Recall that a dg-module $M$ over a dga is called \emph{semifree}, if it has an exhaustive filtration $0=F_{-1}\subset F_0\subset \ldots $ such that all quotients $F_{i+1}/F_i$ are free dg-modules. A mapping cone of   a map $f:M_*\rightarrow N_*$ between free dg-modules is always semifree. The filtration is given by $N_*\subset \Cone(f)_*$ and the quotient is $\Sigma M_*$.
We classified which semifree dg-modules of the form $\Cone(w)_*$ are realizable topologically. It would be interesting to consider the realizability of semifree dg-modules over the Koszul algebra $K_*$ in general.
We believe that this question is even harder and may require (possibly higher) cohomology operations in addition to the product structure in the spectral sequence.
\end{remark}

\bibliographystyle{amsalpha}
\bibliography{bibl}
\end{document}